\documentclass[letterpaper,11pt]{amsart}
\usepackage[margin=2cm]{geometry}
\usepackage{amsmath}
\usepackage{calc}
\usepackage[dvipsnames]{xcolor}
\usepackage{accents}
\usepackage[all,cmtip]{xy}                      %

\usepackage{tikz}

\usepackage{soul}

\CompileMatrices                            

\UseTips                                    

\input xypic
\usepackage[bookmarks=true]{hyperref}       

\usepackage{amssymb,latexsym,amsmath,amscd}
\usepackage{xspace}
\usepackage{color}
\usepackage{graphicx}
\usepackage{dsfont}

\usepackage{enumitem}
\usepackage{hyperref}
\hypersetup{
	colorlinks,
	linkcolor={MidnightBlue},
	citecolor={Green}
}
\usepackage[nameinlink,noabbrev]{cleveref}
\usepackage[normalem]{ulem}
\usepackage{centernot} 

\reversemarginpar

\DeclareMathOperator{\Res}{Res}

\theoremstyle{plain}
\newtheorem{theorem}{Theorem}[section]
\newtheorem*{theorem*}{Theorem}
\newtheorem{proposition}[theorem]{Proposition}
\newtheorem{corollary}[theorem]{Corollary}
\newtheorem{lemma}[theorem]{Lemma}

\theoremstyle{definition}
\newtheorem{definition}[theorem]{Definition}

\newtheorem{notation}[theorem]{Notation}

\newtheorem{remark}[theorem]{Remark}
\newtheorem{example}[theorem]{Example}

\newcommand{\enm}[1]{\ensuremath{#1}}          %
\newcommand{\op}[1]{\operatorname{#1}}
\newcommand{\cal}[1]{\mathcal{#1}}

\newcommand{\CC}{\enm{\mathbb{C}}}

\newcommand{\NN}{\enm{\mathbb{N}}}

\newcommand{\PP}{\enm{\mathbb{P}}}

\newcommand{\TT}{\enm{\mathbb{T}}}

\newcommand{\Bb}{\enm{\cal{B}}}
\newcommand{\Cc}{\enm{\cal{C}}}

\newcommand{\Ee}{\enm{\cal{E}}}
\newcommand{\Ff}{\enm{\cal{F}}}

\newcommand{\Ii}{\enm{\cal{I}}}

\newcommand{\Ll}{\enm{\cal{L}}}

\newcommand{\Oo}{\enm{\cal{O}}}

\newcommand{\Ss}{\enm{\cal{S}}}

\renewcommand{\phi}{\varphi}
\renewcommand{\theta}{\vartheta}
\renewcommand{\epsilon}{\varepsilon}

\newcommand{\sing}{\op{Sing}}

\DeclareMathOperator{\red}{red}
\DeclareMathOperator{\reg}{sm}

\renewcommand{\to}[1][]{\xrightarrow{\ #1\ }}

\newcommand{\old}[1]{}

\date{}
\title{
On the strong base locus of a projective variety}

\author{Edoardo Ballico \and Maria Chiara Brambilla \and Pierpaola Santarsiero}

\newcommand{\Addresses}{{
  \bigskip
  \footnotesize

  \textsc{Edoardo Ballico, Universit\`a di Trento, Dipartimento di Matematica, Via Sommarive 14, 38123 Povo, Trento, Italy }\par\nopagebreak
  \textit{E-mail address}: \email{edoardo.ballico@unitn.it}

  \medskip

\textsc{Maria Chiara Brambilla,Dipartimento di Ingegneria Industriale e Scienze Matematiche, Universit\`a Politecnica delle Marche, Via Brecce Bianche
I-60131 Ancona, Italy}\par\nopagebreak
  \textit{E-mail address}: \email{m.c.brambilla@staff.univpm.it}

  \medskip
  \textsc{Pierpaola Santarsiero, Dipartimento di Ingegneria Industriale e Scienze Matematiche, Universit\`a Politecnica delle Marche, Via Brecce Bianche
I-60131 Ancona, Italy}\par\nopagebreak
  \textit{E-mail address}: \email{p.santarsiero@staff.univpm.it}
}}



\begin{document}
\begin{abstract}
 We introduce and study the base locus and the strong base locus of a projective variety $X$. The base locus of $X$ parametrizes configurations of smooth points of $X$ where the span of the tangent spaces of $X$ at these points intersects $X$ at some additional smooth point. The strong base locus parametrizes configurations of smooth points of $X$ for which the span of the tangent spaces of $X$ at the given configuration contains the entire tangent space at an additional point.
 
These notions  originate from the study of base loci of tangential projections, are strictly related to interpolation problems with double points in special position, and
provide a natural framework to study tangential contact for nongeneral points.
We give first properties and explore connections with Terracini loci and with the concept of identifiability. We focus on tensor-related varieties and characterize the nonemptiness of base loci and strong base loci for Veronese and Segre-Veronese varieties.
\end{abstract}
\maketitle

\section{Introduction}

Let $X \subset \PP^N$ be an integral and nondegenerate variety, meaning that $X$ is irreducible, reduced, and not contained in any proper linear subspace of $\PP^N$. 
For a finite set $A\subset X_{\reg}$ of $r$ smooth points, denote by $\Ii_{A}$ its ideal sheaf and let $2A$ be the zero-dimensional scheme supported at $A$ whose ideal sheaf is $\Ii_{2A}=(\Ii_{A,X})^2$.

If $X$ is embedded by a (possibly incomplete) linear system $\Ll\subseteq |\Oo_X(1)|$, we denote
by $\Ll(-2A)$ the linear system of all divisors in $\Ll$ containing $2A$. Set $\Ll=\PP(V)$ and $\Ll(-2A)=\PP(W)$.
If we denote by $\langle 2A\rangle$ the span of the union of the tangent spaces of $X$ at the points of $A$, the dual projection gives a rational map $$\pi_{\langle 2A\rangle}:\PP^N=\PP(V^\vee)\dashrightarrow  \PP(W^\vee)$$ which is the linear projection of $X$ from the span $\langle 2A\rangle$.
The rational map defined from $\Ll(-2A)$ factors through $\pi_{\langle 2A\rangle}$ and the two maps are identified with a slight abuse of notation. The projection $\pi_{\langle 2A \rangle}$ is usually called the \emph{tangential projection} of $X$ from $A$, see for instance \cite[Section 3]{CC02}.

Generally speaking, one expects the base locus of the subsystem $\Ll(-2A)$, 
i.e. the indeterminacy locus of the tangential projection $\pi_{\langle2A \rangle}$, to be given only by the points of $A$. However, this is not the case in many instances, especially if the points of $A $ are not general. An easy example is given by $\lceil (d+1)/2\rceil$ points lying on a rational normal curve in a Veronese variety $\nu_d(\PP^n)\subset \PP^{\binom{n+d}{d}-1}$, which correspond to collinear points in $\PP^n$ embedded via $\Oo_{\PP^n}(d)$ for some $d\geq 2$.

Hence, one may wonder for which configurations of $r$ smooth points of $X$ 
the base locus of the corresponding subsystem is bigger than the expected one. Equivalently, with a simpler language one may ask when $\langle 2A\rangle$ contains another point $q\in X_{\reg}\setminus A$. This motivated us to introduce the following geometric object. 

Denote by $\Ss_r(X)$ the variety parametrizing unordered sets of $r$ smooth points. 
\begin{definition}\label{def: Bb}
Let $X\subset \PP^N$ be an integral and nondegenerate variety. Define the 
$r$-th \emph{base locus of} $X$ as 
$$
\Bb_r(X):=\{A\in \Ss_r(X) \, \vert \, \langle 2A\rangle \ne \PP^N \hbox{ and }\langle 2A\rangle \cap (X_{\reg}\setminus A)\ne \emptyset  \}.
$$
For a fixed $A\in \Bb_r(X)$, we call $B(A)=(\langle 2A\rangle \cap (X_{\reg}\setminus A))_{\red}$, that is the set theoretic intersection of  $\langle 2A \rangle$ and $X_{\reg}\setminus A$.
\end{definition}

Notice that to every finite set $A$ in $\Bb_r(X)$ it corresponds a set of points $B(A)$ and we remark that $B(A)\cup A$ is the support of the base locus of the system $\Ll(-2A)$, or of the rational projection $\pi_{\langle 2A\rangle}$.

One can go further and wonder for which configurations $A$ of points in $X_{\reg}$ the base locus of the corresponding subsystem $\Ll(-2A)$ actually contains another double point. Geometrically, this corresponds to understanding when $\langle 2A\rangle$ contains the tangent space $T_q X=\langle 2q\rangle$ at some point $q\in X_{\reg}$ distinct from the points of $A$. 
\begin{definition}\label{def: Ee}
Let $X\subset \PP^N$ be an integral and nondegenerate variety. Define the $r$-th \emph{strong base locus of} $X$ as 
$$
\Ee_r(X):=\{ A\in \Ss_r(X) \, \vert \, \langle 2A\rangle\ne \PP^N \hbox{ and }\langle 2(A \cup  \{q\})\rangle =\langle 2A\rangle, \hbox{ for some }q\in X_{\reg}\setminus A\}. 
$$
For a given $A\in \Ss_r(X)$, the \emph{tangential contact locus} $C(A)$ of $A$ is $C(A)=\mathrm{Sing} (\langle 2A \rangle \cap X_{\reg}) $. We have that $A\in \Ee_r(X)$ if and only if $C(A)\setminus A\neq \emptyset$.
We set $E(A)=C(A)\setminus A$, which is Zariski closed in $X_{\reg}\setminus A$. We denote by $\tilde{\Ee}_{r}(X)$ the set of all $A\in \Ee_r(X)$ such that $C(A)$ is infinite, i.e., it is positive dimensional:
$$
\tilde{\Ee}_{r}(X):=\{ A\in \Ss_r(X)\, \vert \, \langle 2A\rangle\ne \PP^N \hbox{ and } \dim (C(A))>0 \}.
$$
\end{definition}
One can also consider slight variations of the loci $\Bb_r(X)$ and $\Ee_r(X)$. For instance, for $\Bb_r(X)$ one can relax the smooth assumption on the extra point of $X$ that lies in $\langle 2A\rangle$. Similarly for $\Ee_r(X)$. However in our manuscript we will always restrict to working on $X_{\reg}$. 
Notice that $\Bb_r(X)$ can also be seen as all the $A\in \Ss_r(X)$ for which $\langle 2A\rangle \ne \PP^N $ and $\langle 2A\rangle =\langle \langle 2A\rangle\cup q\rangle$ for some $q\in X_{\reg}$. 
Hence it is
 immediate to derive the following chain of inclusions
$$
\tilde{\Ee}_r(X)\subseteq \Ee_r(X)	\subseteq \Bb_r(X).
$$
From a computational point of view, in many of the most interesting cases for applications, namely those related to tensor decomposition \cite{landsberg2011tensors, bernardi2018hitchhiker}, 
one knows the equations defining $X$ inside $\PP^N$. This means that for any known $A\in \Ss_r(X)$ the linear spaces $\langle A\rangle$
and $\langle 2A\rangle$ are easy to describe. Hence knowing these equations we get the scheme-theoretic intersection of $X$ and $\langle 2A\rangle$. As a result, it should not be complicated to compute $B(A)$ and $E(A)$, and also to understand if a given $A\in \Ss_r(X)$ is an element of either $\Bb_r(X)$ or $\Ee_r(X)$.

We believe that both the base locus and the strong base locus of a variety are interesting geometric objects. Moreover the strong base locus of a projective variety $X$ has a natural interpretation in the setting of the geometry of higher secant varieties to $X$. 
\subsection*{Geometric interpretation}
In the generic scenario, that is when the points of $A$ are general, the main geometric question motivating the study of $\Ee_r(X)$ is closely related to the concept of tangential weak defectivity first developed in \cite{COtwd}. A variety $X$ is called \emph{tangentially weakly defective} if, for a general set of points, the span of their tangent spaces contains the tangent space at some other point of $X$. 

The notion of tangential weak defectivity is based on the two celebrated Terracini Lemmas. These lemmas are fundamental in the study of secant varieties, a classical topic that remains as relevant and interesting today as it ever was, especially in the context of tensor decomposition \cite{landsberg2011tensors, bernardi2018hitchhiker}. The first Terracini Lemma \cite{terracini1911, aadlandsvik1987joins} says that the dimension of the $r$-secant variety $\sigma_r(X)$ of $X$ equals the dimension of the span $\langle T_{p_1}X, \dots, T_{p_r}X \rangle$ of the $r$ tangent spaces at general points $p_i\in X$. A variety $X$ is \emph{defective} if the dimension of $\sigma_r(X)$ is smaller than $\min\{ r(\dim X + 1)-1, N \}$. 

The second Terracini Lemma \cite{terracini1911} states that when $X$ is defective, a general hyperplane containing the span of the tangent spaces at a general set of $r$ points of $X$ is actually tangent along a positive-dimensional subvariety (see \cite{cilhir,CC02} for modern references). This lemma led to the notion of \emph{weak defectivity}, first introduced in \cite{CC02} as a tool to better understand defectivity, bringing renewed attention to the geometry of contact loci \cite{chiantini2006concept, chiantini2010dimension, ballico2018dimension}. In particular (see for instance also \cite[Lemma 2.3]{BO08}), the second Terracini Lemma says that if $X$ is defective, then there exists a positive-dimensional subvariety $\mathcal{C} \subseteq X$
through $r$ general points $p_1, \dots, p_r\in X$, where the tangent spaces at points of $\mathcal{C}$ lie in the span of the tangent spaces at $p_1, \dots, p_r$. In other words, if $X$ is defective then it is tangentially weakly defective. Note that the reverse implication does not hold (see for instance \cite[Example 2.5]{COtwd}). 
If $X$ is $r$-tangentially weakly defective, then already for a general $A\in \Ss_r(X)$ we have $A\in \tilde{\Ee}_r(X)$. 
While if $X$ is not $r$-weakly defective then for a general $A\in \Ss_r(X)$ we have $A\notin \Ee_r(X)$ \cite[Theorem 1.4]{CC02}. These immediate consequences make clear the connections of these notions with the strong base loci we just introduced and aim to investigate. We might interpret the understanding of $\Ee_r(X)$ as the study of nongeneric tuples that have a tangential contact locus.

 We point out that in the literature related to contact loci (see for instance \cite{CC02, chiantini2006concept, chiantini2010dimension, CasarottiMella2022, LafaceMassarentiRischter2022, FreireCasarottiMassarenti2021}), only a general
$S\in \Ss_r(X)$ is considered, whereas in our setting $S$ is arbitrary. The corresponding considered contact locus $\Cc_S$ of the general element in $\Ss_r(X)$ is defined as the union of the closure in $X$ of the irreducible components of $\mathrm{Sing}(\langle 2S\rangle \cap X_{\reg})$ containing at least one point of $S$. Hence the contact locus $C(S)$ of \Cref{def: Ee} coincides with $S$ under that definition if and only if each point of $S$ is an isolated point of $\mathrm{Sing}(\langle 2S\rangle \cap X)$. In particular, since only general $S$ are considered there, if at least one point of $S$ is isolated, then all of them are isolated.

In this manuscript we begin our investigation by focusing on Veronese and Segre–Veronese varieties because of  their connection to tensor decomposition and tensor identifiability which gives further evidence of the interest in studying the introduced loci. Indeed, motivated by tensor identifiability, the relevance of determining whether a tuple of points lies in $\Ee_r(X)$ already appears in \cite{COV17} with a different language, we refer to \Cref{section: identifiability} for further discussion.

Still, the behavior of both the base locus and the strong base locus of a projective variety in greater generality remains largely unexplored and naturally call for future exploration.

\subsection*{Detailed outline and main results}
In \Cref{section: prel} we give the needed preliminaries as well as explore first properties of $\Bb_r(X)$ and $\Ee_r(X)$. In particular, we reformulate both definitions in terms of interpolation theory (\Cref{remark: traduzione Bb ed Ee cohomologica}). In the same section we introduce the \emph{critical scheme for a pair} (\Cref{def: critical scheme for a pair}), which is our main technical tool used to study these loci. We conclude with \Cref{connection terracini}, where we compare our new objects with Terracini loci. \Cref{example: confronto terracini con Bb ed Ee} shows that neither $\Bb_r(X)$ nor $\Ee_r(X) $ or $\tilde{\Ee}_r(X)$ is comparable with respect to inclusion with the $r$-Terracini locus of $X$. However starting from elements of $\Bb_r(X)$ one can construct elements of the $(r+1)$-Terracini locus of $X$, see \Cref{remark: dato A in Ee costruisco sempre una r+1 upla terracini}.\\
\Cref{section: veronese} focuses on $X$ being a Veronese variety. We sum up below our main results.
\begin{enumerate}
    \item Let $X$ be an $n$-dimensional Veronese variety embedded in degree $d$. Then: \\
if $n=1$ we have $\Bb_r(X)=\Ee_r(X)=\emptyset$ (\Cref{example: Bb vuoto crn}); otherwise for $n\geq 2$ we have
\begin{itemize}
    \item  $\Bb_r(X)$ is nonempty if and only if $2r\ge d+1$ (\Cref{theorem: Bb per veronesi});
    \item $\Ee_r(X)$ and $\tilde{\Ee}_r(X)$ are nonempty if and only if $r\geq d $ (\Cref{theorem: Ee e tilde Ee per veronesi});
    \item for $r=d$ we have $\Ee_d(X)=\tilde{\Ee}_d(X)=\{ A\, |\, A \text{ is contained in a degree-$d$ rational normal curve}\}$ (\Cref{theorem: r=d veronesi Ee}).
\end{itemize} 
\end{enumerate}
The focus of \Cref{section: segre-veronese} is on Segre-Veronese varieties. Our main results are the following.
\begin{enumerate} \setcounter{enumi}{1}
    \item
Let $X$ be a $k$-factors Segre-Veronese variety embedded in degree $d_1,\dots,d_k\geq 1$, for some $k\geq 2$ and let $t:= \min\{d_1,\dots,d_k\}$. 
Then 
\begin{itemize}
    \item $\Bb_r(X)$ is nonempty if and only if $2r\ge t+1$ (\Cref{theorem: Bb per SV});
    \item $\Ee_r(X)$ and $\tilde{\Ee}_r(X)$ are nonempty if and only if $r\geq t+1 $ (\Cref{theorem: Ee e tilde Ee per segre-veronesi}).
\end{itemize} 
\end{enumerate}

Moreover, both for Veronese varieties and Segre-Veronese varieties we make a detailed study of the first cases in which $\Bb_r(X)$ is not empty classifying the base locus $B(A)$ of any given $A\in \Bb_r(X)$.

A useful tool to prove our results is \Cref{lemma: famoso per SV}, which may be of independent interest since it can be thought as a direct generalization in the Segre-Veronese setting of the well-known \cite[Lemma 34]{BGI11}.
Lastly, \Cref{section: identifiability} explores the connection of $\Ee_r(X)$ with the concept of tensor identifiability.

\subsection*{Acknowledgements}
EB and MCB are members of INdAM GNSAGA. MCB and PS acknowledge support from the European Union under NextGenerationEU. PRIN 2022, Prot. 2022E2Z4AK and PRIN 2022 SC-CUP: I53C24002240006.

\section{Preliminaries and first properties of $\Bb_r(X)$ and $\Ee_{r}(X)$}\label{section: prel}
In the following, we work over $\CC$. We recall here the notation we will use for zero-dimensional schemes and in particular for two-fat points.
\begin{notation}
For each closed and reduced algebraic subset $S\subset X$ let $2S$ denote the closed subscheme of $X$ with $(\Ii_{S,X})^2$ as its ideal sheaf. Note that $2S$ contains $2q$ for all $q\in S$ and that $2S$ is the minimal closed subscheme of $S$ with these properties.

Given a zero-dimensional scheme $Z\subset X$, in the following we will sometimes use the notation $Z_{\red}$ to denote the support of $Z$. For a simple point $p\in Z_{\red}$ we will use the notation $p\subset Z$ to consider $\{ p\}$ as a connected component of $Z$ and also $p\cup Z$ to consider the schematic union $\{ p\} \cup Z$.

For any scheme $Z\subset \PP^N$ let $\langle Z \rangle \subset\PP^N$ denote the minimal linear subspace of $\PP^N$
containing $Z$.
\end{notation} 
Let us pass now to recall some general standard facts on the cohomology of zero-dimensional schemes. 

\begin{remark}\label{remark: algebra lineare per schemini di grado at most 2}
    Let $Z\subset \PP^N$ be a zero-dimensional scheme such that every connected component of $Z$ has degree at most 2. Then $\dim \langle Z\rangle\leq \deg(Z)-1$. 
    
    Moreover if $\dim \langle Z\rangle<\deg(Z)-1$ then there exists a proper subscheme $W\subsetneq Z$ with $ \deg(W)=\dim \langle Z\rangle+1$ and $\dim \langle W\rangle=\deg(W)-1$. In particular, this gives $\langle W\rangle=\langle Z\rangle$.
\end{remark}

\begin{remark}\label{succ esatta}
Fix an integral projective variety $X$, a zero-dimensional scheme $Z\subset X$ and a point $p\in X_{\reg}$ with $p\not\in Z_{\mathrm{red}}$. The long cohomology exact sequence of the exact sequence
$$0\rightarrow \Ii_{Z\cup p,X}(1)\rightarrow \Ii_{Z,X}(1)\rightarrow \Oo_{p}\rightarrow 0$$
gives
$$H^0(\Ii_{Z\cup p,X}(1))\rightarrow H^0( \Ii_{Z,X}(1))\rightarrow \CC\rightarrow H^1(\Ii_{Z\cup p,X}(1))\rightarrow H^1( \Ii_{Z,X}(1)).$$
More generally, for any zero-dimensional scheme $W\subset Z$ we have
$$0\rightarrow \Ii_{Z,X}(1)\rightarrow \Ii_{W,X}(1)\rightarrow \Oo_{Z\setminus W}\rightarrow 0$$
from which we get
$$h^0(\Ii_{Z,X}(1))\le h^0(\Ii_{W,X}(1)) \le h^0(\Ii_{Z,X}(1))+\deg(Z)-\deg(W)$$ and $$h^1(\Ii_{W,X}(1))\le h^1(\Ii_{Z,X}(1))\le
h^1(\Ii_{W,X}(1))+\deg(Z)-\deg(W).$$ 
\end{remark}

\begin{remark}\label{prel01}
Fix an integral projective variety $X$ 
and a zero-dimensional scheme $Z\subset X$. Let $\Ff$ be a rank $1$ torsion-free coherent sheaf on $X$ which is locally free in a neighborhood of $Z$. Note that $\deg(\Ff_{|Z}) =\deg(Z)$. Since $\dim Z =0$, we have 
$h^i(Z,\Ff_{|Z})=0$ for all $i\ge 1$. The long cohomology exact sequence induced by $$0\to \Ii_Z\otimes \Ff\to \Ff\to \Ff_{|Z}\to 0$$ 
gives $h^i(\Ii_Z\otimes \Ff) =h^i(\Ff)$ for all $i\ge 2$
and $$h^1(\Ii_Z\otimes \Ff) =h^1(\Ff) +\deg(Z) -h^0(\Ii_Z\otimes \Ff) +h^0(\Ff).$$ 
In particular when $\Ff=\Oo_X(1)$, we have 
$h^i(\Ii_{Z,X}(1))=h^i(\Oo_X(1))$ for $i\geq 2$ and
$$h^0(\Oo_X(1))-h^0(\Ii_{Z,X}(1))= \deg(Z)-(h^1(\Ii_{Z,X}(1))-h^1(\Oo_X(1))).$$
Moreover if $X\subset \PP^N$ is embedded by a complete linear system, we have
$$\dim \langle Z\rangle = \deg(Z)-(h^1(\Ii_{Z,X}(1))-h^1(\Oo_X(1)))-1.$$
\end{remark}

There is another interpretation of the loci $\Bb_r(X),\Ee_r(X)$ and $\tilde{\Ee}_r(X)$ in terms of cohomology of zero-dimensional schemes of two-fat points when $X$ is embedded in $\PP^N$ via a complete linear system. 

\begin{remark}\label{remark: traduzione Bb ed Ee cohomologica}
Assume that $X\subset \PP^N$ is embedded via a complete linear system. Let $A\in \Ss_r(X)$ and assume $A\in \Bb_r(X)$. Saying that $\langle 2A \rangle\neq \PP^N$ is equivalent to say that $h^0(\Ii_{2A,X}(1))\neq 0$, while requesting that $ \langle 2A \rangle=\langle 2A\cup p\rangle$ for some $p\in X_{\reg}\setminus A$ is equivalent to ask that $h^0(\Ii_{2A,X}(1))=h^0(\Ii_{2A\cup p,X}(1))$. 
Moreover, we have
\begin{equation}\label{cohom}
\langle 2A\rangle = \langle 2A \cup p\rangle \quad\Leftrightarrow\quad h^0(\Ii_{2A\cup p,X}(1))=h^0(\Ii_{2A,X}(1)) \quad\Leftrightarrow\quad h^1(\Ii_{2A\cup p,X}(1))=h^1(\Ii_{2A,X}(1))+1,
\end{equation}
where the second equivalence follows by applying \Cref{prel01} with $\Ff=\Ii_{2A}(1)$ and $Z=p$. Hence we reformulate the objects of Definitions \ref{def: Bb} and \ref{def: Ee} as:
\begin{align*}
\begin{split}
\Bb_r(X)&=\left\{ A\in \Ss_r(X) \, \vert \, h^0(\Ii_{2A,X}(1))=h^0(\Ii_{2A\cup p,X}(1))\neq 0,  \hbox{ for some } p\in X_{\reg}\setminus A \right\} 
\end{split}\\
&=\big\{ A\in \Ss_r(X) \, \vert \, h^0(\Ii_{2A,X}(1))\neq 0, \text{ and }
h^1(\Ii_{2A\cup p,X}(1))=h^1(\Ii_{2A,X}(1))+1 \hbox{ for some } p\in X_{\reg}\setminus A \big\},\nonumber
\end{align*}
and
\begin{align*}
\begin{split}
\Ee_r(X)&=\{A \in \Ss_r(X) \, \vert \, h^0(\Ii_{2A,X}(1))=h^0(\Ii_{2(A\cup p),X}(1))\neq 0,  \hbox{ for some } p\in X_{\reg}\setminus A \}
\end{split}\\
&=\big\{  A\in \Ss_r(X) \, \vert \,  h^0(\Ii_{2A,X}(1))\neq 0 ,\, h^1(\Ii_{2(A\cup p),X}(1))=h^1(\Ii_{2A,X}(1))+\dim X+1 \hbox{ for some } p\in X_{\reg}\setminus A \big\}. \nonumber
\end{align*}
A similar reformulation clearly holds also for $\tilde{\Ee}_r(X)$. 
\end{remark}

To lighten the notation, when it is clear from the context in the following, we will often omit $X$ when dealing with $\Ii_{2A,X}$ and simply write $\Ii_{2A}$.

 If we are able to identify a portion $T$ of the tangential contact locus for a set of $r$ regular points, then belonging to $\tilde{\Ee}_r(X)$ can be recast as follows.
\begin{remark}\label{remark: reformulation of tilde eps with contact}
Let $A\in \Ss_r(X)$ such that $\langle 2A\rangle \ne \PP^N$ and fix a closed set $T\subset X$ such that each irreducible components are positive-dimensional and intersect $X_{\reg}$. The following conditions are equivalent:
\begin{enumerate}[label=\roman*)]
\item\label{ref tilde eps 1} $A\in \tilde{\Ee}_r(X)$ with contact locus containing $T$;
\item\label{ref tilde eps 2}  for all  finite $E\subset T\cap X_{\reg}$ we have $H^0(\Ii_{2A\cup 2E}(1)) =H^0(\Ii_{2A}(1))$;
\item\label{ref tilde eps 3} for all $q\in T\cap X_{\reg}$ we have $\langle 2q\cup 2A\rangle =\langle 2A\rangle$.
\end{enumerate} 
Indeed, notice that  \ref{ref tilde eps 1} implies \ref{ref tilde eps 2} and \ref{ref tilde eps 2} implies \ref{ref tilde eps 3}, by  the definition of $\tilde{\Ee}_r(X)$.
Moreover,  \ref{ref tilde eps 3} implies \ref{ref tilde eps 1}. 
\end{remark}

We now pass to describe basic properties and first examples of the loci just introduced.
There are examples of nonempty $\Bb_1(X)$ as shown in the following.
\begin{example}\label{example: B is not empty if X is ruled by lines}
    Let $X\subsetneq \PP^N$ be such that there exists a line $L\subset X$ with $L\cap X_{\reg}\neq \emptyset$. Then every point of $ X_{\reg}\cap L $ is a point of $\Bb_1(X)$.  
\end{example}
Notice that since $X$ is integral and nondegenerate, if $\dim X=1$, then $\tilde{\Ee}_r(X)=\emptyset$ for all values of $r$.
On the other hand, $\Bb_1(X)$ and $\Ee_1(X)$ may be not empty. For example, since a plane quartic $X$  admits bitangents, we see that $\Ee_1(X)\neq\emptyset$ (see also the forthcoming \Cref{example: confronto terracini con Bb ed Ee}).

In the following remark we show how to construct a set of points  belonging to $\Bb_{r+s}(X)$ (resp. $\Ee_{r+s}(X)$) starting from
elements of $\Bb_r(X)$ (resp. $\Ee_r(X)$). 
\begin{remark}\label{remark: a partire da punti in Bb ed Ee si costruiscono ancora cose in Bb ad Ee con ip h0 non nullo}
Fix integers $s,r>0$ and let $A\in \Bb_r(X)$. Take a set  $F\in \Ss_{s}(X) $ with $F\cap A=\emptyset$, $F\cap B(A)=\emptyset$ and such that $\langle 2(A\cup F)\rangle\neq \PP^N$. Then $A\cup F \in \Bb_{r+s}(X)$ and obviously $B(A\cup F)\supseteq B(A)$. Notice that the assumption $F\cap B(A)=\emptyset$ guarantees that the points of $F$ contribute new directions within $\langle 2(A\cup F)\rangle$. A similar reasoning applies to $\Ee_r(X)$. More precisely, given $A\in \Ee_r(X)$, then for any $F\in \Ss_s(X)$ with $F\cap A=\emptyset$, $F\cap E(A)=\emptyset$ and such that $\langle 2(A\cup F)\rangle\neq \PP^N$ we have that $A\cup F\in \Ee_{r+s}(X)$. 
\end{remark}

The assumptions of \Cref{remark: a partire da punti in Bb ed Ee si costruiscono ancora cose in Bb ad Ee con ip h0 non nullo} can be weakened for $\tilde{\Ee}_{r}(X)$, as we show in the next remark, where we prove that
if $\tilde{\Ee}_r(X)\neq \emptyset$, then for arbitrary $s\ge 0$ we have $\tilde{\Ee}_{r+s}(X)\neq \emptyset$.

\begin{remark}\label{remark: una volta in Ee se aggiungo punti resto in Ee}
Fix integers $s,r>0$ and let $A\in \tilde{\Ee}_{r}(X)$. If we take any subset $F$ in $E(A)$ of arbitrary cardinality $s$, then $A\cup F\in \tilde{\Ee}_{r+s}(X)$. Moreover, the contact locus $C(A)$ of $A$ is contained in the contact locus $C(A\cup F)$ of $A\cup F$. 
\end{remark}

Before proceeding further, it is important to introduce a technical tool that will be central to our analysis in the rest of the manuscript.  
\begin{definition}\label{def: critical scheme for a pair}
Fix a set $A$ of smooth points of $X$ and a smooth point $q$ of $X$ with $q\not\in  A$. 
A \emph{critical scheme for the pair} $(A,q)$ is a zero-dimensional scheme $Z\subset X$ such that
\begin{enumerate}[label=\roman*)]
    \item $q\in Z_{\red}\subseteq 2A\cup \{ q\}$, 
\item for any connected component $v \subset Z$, $\deg(v)\leq 2$,
\item\label{def critical span Z} $q \in \langle Z\setminus q\rangle$, 
\item\label{def critical minimality} $q\not\in\langle W\rangle$ for any $W\subsetneq Z\setminus q$ (\emph{minimality condition}).
\end{enumerate}
In other words, a critical scheme for the pair \((A, q)\) is a zero-dimensional subscheme \(Z \subset X\), supported in \(A \cup \{q\}\), that contains \(q\) as a simple point, such that each component of \(Z\) has degree at most 2, and \(Z\) is minimal with respect to spanning \(q\); that is $q$ lies in the linear span of $Z \setminus q$, but not in the span of any proper subscheme of $Z\setminus q$.
\end{definition}

\begin{lemma}\label{lemma: properties critical scheme} 
Let $Z\subset X$ be a critical scheme for a couple $(A,q)$. Then $\dim \langle Z\rangle =\deg(Z) -2$, $q\in \langle Z\setminus \{q\}\rangle$ and $\dim \langle Z'\rangle = \deg(Z')-1$ for all $Z'\subsetneq Z$.
\end{lemma}
\begin{proof}
First notice that since $Z$ is a critical scheme $\langle Z\rangle=\langle Z\setminus q\rangle \supseteq \langle W\cup q \rangle $ for all $W\subset Z\setminus q$.  Assume by contradiction that $\dim \langle  Z\setminus q  \rangle<\deg(Z\setminus q)-1$. By \Cref{remark: algebra lineare per schemini di grado at most 2} there exists $U\subset Z\setminus q$ with $\deg(U)=\dim \langle Z\setminus q \rangle+1 $ such that $\dim \langle U \rangle=\dim \langle Z\setminus q\rangle=\langle Z \rangle \supseteq \langle U\cup q \rangle$. But this is impossible by the minimality condition of $Z$.
\end{proof}

\begin{remark}\label{remark: riformulazione proprietà schema critico in cohomologia per casi belli}
Assume that $X$ is embedded in $\PP^N$ via a complete linear system.
    \Cref{def critical span Z} of \Cref{def: critical scheme for a pair} says that $\langle Z \rangle=\langle Z\setminus q\rangle$ and, from a cohomologial point of view, this is equivalent to say that $h^1(\Ii_{Z}(1))=h^1(\Ii_{Z\setminus q}(1))+1$ (cf. also \Cref{remark: traduzione Bb ed Ee cohomologica}). The minimality property \ref{def critical minimality} says that $\langle W \rangle\neq \langle W\cup q\rangle$ for all $W\subset Z\setminus q$, and this is translated as $ h^1(\Ii_{W}(1))=h^1(\Ii_{W\cup q}(1))$.
Hence for a critical scheme $Z$ we have $h^1(\Ii_Z(1))\geq 1$ and it is easy to prove that actually $h^1(\Ii_Z(1))=1$ because by \Cref{lemma: properties critical scheme} we have $h^1(\Ii_{Z\setminus q}(1))=0$. Moreover, since $h^1(\Ii_W(1))\leq h^1(\Ii_{Z\setminus q}(1))$ for all $W\subset Z\setminus q$, we have  $h^1(\Ii_{W}(1))=h^1(\Ii_{W\cup q}(1))=0$.
\end{remark}

We remark that a simpler notion of critical scheme already appears in the context of Terracini loci (see \cite[Definition 2.10]{BallicoBrambilla2024}). In that case, Chandler curvilinear lemma \cite{Chandler1994,Chandler2000} immediately implies that any Terracini set admits a critical scheme. In our situation, 
it is not straightforward to apply the curvilinear lemma to prove that a critical scheme for a couple always exists for any set $A\in \Bb_r(X)$. Therefore we prove it in the  next lemma.

\begin{lemma}\label{lemma: existence critical scheme characterization}
    Fix a set $A$ of $r$ smooth
    points of $X\subset \PP^N$ such that $\langle 2A\rangle \ne \PP^N$ and let $q\in X_{\reg}\setminus A$. A critical scheme for the pair $(A,q)$ exists
if and only if $q\in \langle 2A\rangle$, i.e. if and only if $A\in \Bb_r(X)$ and $q\in B(A)$. 
\end{lemma}
\begin{proof}
If $(A,q)$ has a critical scheme $Z$, by definition we have $q\in \langle Z\setminus q\rangle\subseteq \langle 2A\rangle$.

Now assume $q\in \langle 2A\rangle$. Let $F\subseteq A$ be a minimal subset of $A$ such that $q\in \langle 2F \rangle $ and notice that $F\ne \emptyset$. Fix $p\in F$ and call $G=F\setminus \{p\}$ where now $G$ might be empty.
By the minimality of $F$ we know that $q\in \langle 2F \rangle$ and $q\notin \langle 2G\rangle$, so in particular $\dim \langle 2G\cup q\rangle =\dim \langle 2F\rangle+1 $. Hence $ \langle q\cup 2F\rangle\cap \langle 2p\rangle\neq \emptyset$. 

The linear space $\langle 2p\rangle$ is the union of lines through $p$ contained in $\langle 2p\rangle$.
Hence there is a minimal $v\subset 2p$ with $\deg(v) \le 2$ such that $\langle v\rangle \cap \langle 2G\rangle =\emptyset$ and $\langle v\rangle \cap \langle 2G\cup q\rangle \ne \emptyset$.

In particular we have $\dim \langle  v \cup 2G\rangle = \dim \langle  2G\rangle +\deg(v)$.  To conclude, we observe that we still have that $q\in \langle v  \cup 2G   \rangle$ for the following reason. Since $\langle v\rangle \cap \langle 2G\cup q \rangle\neq \emptyset $ we have
\begin{align*}
 \dim \langle 2G\cup q\cup v\rangle& =\dim \langle  2G\cup q\rangle+\dim \langle v\rangle-\dim (\langle v\rangle \cap \langle 2G\cup q\rangle )\\
 &\leq \dim \langle 2G\rangle+1+\deg(v)-1=\dim\langle v\cup 2G\rangle .   
\end{align*}
The procedure we just described is fundamental for the construction of a critical scheme for a couple. Indeed, to construct a critical scheme for the couple $(A,q) $, it suffices to iterate the above construction as we now explain.

Assume for the moment that $\#F=1$, so that $G=F\setminus \{ p\}=\emptyset$. Applying the procedure just described, we get that the scheme $Z:= v$ is a critical scheme for the couple $(A,q)$ and in particular, since $q\notin F$ we have that $\deg(v)=2$. Otherwise $\#F\geq 2$. We first apply the above procedure to $G=F\setminus \{ p\}$. This gives that $q\in \langle G\cup v\rangle$. Now, since $\#F\geq 2$, we have that $G\neq \emptyset$, so we can take $p_1\in G$, set $G'=G\setminus \{p_1 \}$ and apply again the above procedure. Hence $q\in \langle  G'\cup v \cup v_1\rangle$ for some $v_1$ with $\deg(v_1)\leq 2$. If $\#F=2$ then we have that the zero-dimensional scheme $Z:=v\cup v_1 $ is a critical scheme for $(A,q)$, otherwise we can keep going in the same manner.
\end{proof}

\subsection{Connection with Terracini loci}\label{connection terracini}

We conclude this section by exploring the relation of the objects $\Bb_r(X),\Ee_r(X)$ and the $r$-th Terracini locus of a variety \cite{BCterracini, BBS23}, whose definition we recall here.
\begin{definition}\label{def: terracini locus}
   Let $X\subset \PP^N$ be integral and nondegenerate. The $r$-th \emph{Terracini locus of $X$} is
$$
\TT_r(X)=\{ A\in \Ss_r(X)\, \vert \, \langle 2A\rangle\neq \PP^N \text{ and }\dim \langle 2A\rangle< r(\dim X+1)-1 \}.
$$

\end{definition}
A priori there is no apparent relation between $\TT_r(X) $ and $\Bb_r(X)$ or $\Ee_r(X)$ and one can construct examples of  $A\in \TT_r(X) $ that are not in $\Bb_r(X)$, as well as examples of $A$ in either $\Bb_r(X)$ or $\Ee_r(X)$ but not in $\TT_r(X)$ as shown in the following.

\begin{example}\label{example: confronto terracini con Bb ed Ee}
Let us exhibit explicit examples proving the following statements:
$$
A\in \TT_r(X) \ \ \ \substack{ \overset{\ref{a}}{\centernot\implies} \\ \overset{\ref{b}}{\centernot\Longleftarrow}} \ \ \ 
A\in \Bb_r(X) \text{ and } A\in \TT_r(X) \ \ \ \substack{ \overset{\ref{c}}{\centernot\implies} \\ \overset{\ref{d}}{\centernot\Longleftarrow}} \ \ \ 
A\in \tilde{\Ee}_r(X).
$$
\begin{enumerate}[label=\alph*)]
    \item\label{a} Let $\Cc \subset \PP^2$ be a general curve of degree 4. It is classically known that $\Cc$ has 28 bitangents \cite{MR1578003} so in particular $\TT_2(\Cc)\neq \emptyset$. However for each $A\in \TT_2(\Cc)$, the bitangent $\langle 2A \rangle$ does not intersect $\Cc$ in other points, so $A\notin \Bb_2(\Cc)$.
    \item\label{b} For a plane curve $\Cc\subseteq \PP^2$ of degree $d\geq 3$ we have $\Bb_1(\Cc)=\Ss_1(\Cc)$ while $\TT_1(\Cc)=\emptyset$. Another easy example can be constructed as follows. Let $A\notin \TT_2(\mathcal{C})$ for some smooth irreducible degree 5 curve $\Cc\subset \PP^4$ and notice that such an $A$ always exists (cf. for instance \cite[Proposition 4.1]{GSTT}). 
    Then $\langle 2A\rangle$ is a hyperplane of $\PP^4$ intersecting the curve in 5 points counted with multiplicity, hence there exists a point $p\in \mathcal{C}_{\reg}$ with $p\notin  A$ such that $p\in\langle 2A \rangle  $, so that $A\in \Bb_2(\Cc)$. In particular, notice that $A\in \Bb_2(\Cc)\setminus \Ee_2(\Cc)$.  
    \item\label{c} Since for a variety $X\subset \PP^N$ we have $\tilde{\Ee}_r(X)\subseteq \Ee_r(X) \subseteq \Bb_r(X) $, \cref{a} provides an example of an element in $\TT_r(X) $ that does not lie in $\tilde{\Ee}_r(X)$.
    \item\label{d} Let $Y\subset \PP^N \subset \PP^{N+1}$ be an irreducible nondegenerate smooth variety in $\PP^N$ and fix a point $q\in \PP^{N+1}\setminus \PP^N$. Let $X=\cup_{y\in Y}\langle y,q\rangle$ be the cone over $Y$ with vertex $q$. Notice that for every point $p\in X\setminus \{ q\} $ the line $\ell_p=\langle p,q \rangle $ is included in $X$ and moreover, since $Y$ is smooth, every point of $\ell_p\setminus \{q \}$ is smooth. 
    Now, for each $x\in \ell_p\setminus \{ q\}$ the tangent space $T_xX$ contains $\ell_p$ and this means that $\ell_p\setminus \{ q \}$ is contained in the contact locus of $\{ x\}$. Hence we just proved that $\tilde{\Ee}_1(X)\neq \emptyset$ while clearly $\TT_1(X)=\emptyset$.
    The previous argument can easily be generalized, making any cone $X$ a straightforward example of a variety for which $\tilde{\Ee}_r(X)\neq \emptyset$ for any positive integer $r$. 
 \end{enumerate}

\end{example}

However, starting from sets in $\Bb_r(X)$ one can easily construct $(r+1)$-Terracini sets.
\begin{remark}\label{remark: dato A in Ee costruisco sempre una r+1 upla terracini}
     If $A\in \Bb_r(X)$ then for all $p\in B(A)$ we have $p\in \langle 2A\rangle$, so the tangent space $T_pX$ and the linear space $\langle 2A\rangle$ intersect nontrivially, hence if $\langle2(A\cup p)\rangle\neq\PP^N$ we have $A\cup \{ p\}\in \TT_{r+1}(X)$.

Moreover,  if $ A\in \Ee_r(X)$ then $A\cup \{ p\}\in \TT_{r+1}(X)$ for all $p\in C(A)\setminus A$.
\end{remark}

\section{Veronese varieties}\label{section: veronese}
The purpose of this section is to better understand the objects $\tilde{\Ee}_r(X),\Ee_r(X),\Bb_r(X)$ when $X$ is a Veronese variety. For this part we will use the following notation. Fix integers $n,d\geq 1$, set $N=\binom{n+d}{d}-1$ and let $\nu_d:\PP^n\rightarrow \PP^N$ be the $d$-th Veronese embedding of $\PP^n$ via the complete linear system of divisors of degree $d$. We denote by $V_n^d=\nu_d(\PP^n)$.

As a first instance of computation let us treat separately the case $n=1$ of rational normal curves. The next example shows that $\Bb_r(V^d_1)$ is always empty and we will prove it explicitly, although it could also be easily derived using \Cref{remark: dato A in Ee costruisco sempre una r+1 upla terracini}. 
\begin{example}[Rational normal curves]\label{example: Bb vuoto crn}
   Let $V^d_1\subset \PP^d$ be a rational normal curve of degree $d\geq 2$ and let $A\in \Ss_r(\PP^1)$. Assume by contradiction that $\nu_d(A)\in\Bb_r(V^d_1)$, and let $p\in \PP^1\setminus A$ such that $\nu_d(p)\in B(\nu_d(A))$. By \Cref{remark: traduzione Bb ed Ee cohomologica}, on one hand we have
 $h^0(\Ii_{2A\cup p,\PP^1}(d))>0$, on the other hand  $h^1(\Ii_{2A\cup p,\PP^1}(d))\ge 1$.
 This is a contradiction because $h^i(\Ii_{2A\cup p,\PP^1}(d))=h^i(\Oo_{\PP^1}(d-(2r+1)))$.
\end{example}

The following result is a well-known useful tool.
\begin{lemma}\label{lemma: famoso}\cite[Lemma 34]{BGI11}
Fix integers $n\ge 1$ and $d\ge 1$. Let $Z\subset \PP^n$ be a zero-dimensional scheme such that $\deg(Z)\le 2d+1$ and $h^1(\Ii_Z(d)) >0$. Then there is a line $L\subseteq \PP^n$ such that $\deg(Z\cap L)\ge d+2$.
\end{lemma}

\begin{remark}\label{example: se ci sono abbastanza punti su una retta questa è nel base locus}
Fix a line $L\subset \PP^n$, $n\ge 2$, and an integer $d\ge 2$. Let $A\subset \PP^n$ be a finite set such that $\#(A\cap L)\ge \lceil (d+1)/2\rceil$. Since $\deg(2A\cap L) =2\#(A\cap L)\ge d+1$, $L$ is contained in the base locus of $\Ii_{2A}(d)$. 
\end{remark}

\begin{lemma}\label{lemma: dato A in Bb esiste retta contenente un po di pt}
    Let $n,d\geq 2$, $r\leq d$. Let $A\in \Ss_r(\PP^n)$ and assume $\nu_d(A)\in \Bb_r(V^d_n)$. Fix $q\in \PP^n$ such that $\nu_d(q)\in B(\nu_d(A))$. Then, there exists a line $L\subset \PP^n$ such that $\#(A\cap L)\geq \lceil (d+1)/2 \rceil$ and $ q\in L$.
\end{lemma}
\begin{proof}
   Take $A\in \Ss_r(\PP^n)$ with $\nu_d(A)\in \Bb_r(V^d_n)$ and fix $q\in \PP^n\setminus A$ such that $\nu_d(q)\in B(\nu_d(A))$. Let $Z\subset \PP^n$ be a zero-dimensional scheme such that $\nu_d(Z)$ is a critical scheme for the pair $(\nu_d(A),\nu_d(q))$; such a $Z$ always exists by \Cref{lemma: existence critical scheme characterization}.
Hence $\deg(Z)\le 2r+1\leq 2d+1$ and $h^1(\Ii_Z(d)) >0$. 
\Cref{lemma: famoso} gives the existence of a line $L$ such that $\deg(Z\cap L)\ge d+2$.
Since $q$ is a connected component of $Z$ and all connected components of $Z$ have degree at most $2$ we get $\#(L\cap A)\ge \lceil (d+1)/2\rceil$.

We now want to prove that $Z\subset L$ because this would give in particular that $q\in L$. Note that $h^1(\Ii_{Z\cap L}(d))=h^1(\Ii_{Z\cap L,L}(d))$, since $L=\PP^1$ is arithmetically Cohen-Macaulay, and $h^1(\Ii_{Z\cap L,L}(d))\ge 1$ since $\deg(Z\cap L)\ge d+2$. By \Cref{lemma: properties critical scheme} (cf. also \Cref{remark: riformulazione proprietà schema critico in cohomologia per casi belli}) we know that $h^1(\Ii_{Z'}(d))=0$ for any $Z'\subsetneq Z$. Hence we must have $Z\cap L= Z$, or equivalently that $Z\subset L$. 
\end{proof}

Our next result is a characterization of emptiness for $\Bb_r(V^d_n)$, along with a detailed description of the first cases in which $\Bb_r(V^d_n)$ is not empty, see also \Cref{fig:placeholder}.
\begin{theorem}\label{theorem: Bb per veronesi}
Fix integers $n\ge 2$ and $d\ge 2$. Then 
$$
\Bb_r(V_n^d)\neq \emptyset \text{ if and only if }2r\ge d+1.
$$

Assume $\lceil (d+1)/2\rceil\leq r \leq d$ and let  $A\in \Ss_r(\PP^n)$ be such that $\nu_d(A)\in \Bb_r(V^d_n)$. 
\begin{itemize}
\item\label{description Bb: r=d odd at least 5}
If $r=d\ge5$ and $d$ is odd, then either there is a line $L\subset \PP^n$ containing at least $(d+1)/2$ points of $A$ and such that 
$B(\nu_d(A)) =\nu_d(L\setminus  A)$, or there are two distinct lines $L,R\subset \PP^n$ intersecting in a unique point of $A$, each containing $(d+1)/2$ points of $A$, and such that $B(\nu_d(A)) =\nu_d((L\cup R)\setminus A)$.
\item\label{description Bb: r=d even} If either $r<d$, or $r=d$ and $d$ is even, then there is a line $L\subset \PP^n$ that contains at least $\lceil(d+1)/2\rceil$ points of $A$ and such that 
$B(\nu_d(A)) =\nu_d(L\setminus  A)$.
\item\label{description Bb: r=d=3} Otherwise $r=d=3$ and $\Bb_3(V^d_n)=\Ss_3(V^d_n)$. If $A\in \Ss_3(\PP^n)$ is general then $B(\nu_d(A))\cup \nu_d(A)$ is given by the three distinct lines passing each of them through two points of $A$. Otherwise $A\subset L$ for some line $L\subset \PP^n$ and $B(\nu_d(A))=\nu_d(L\setminus A)$. 
\end{itemize}
\end{theorem}

\begin{proof}
By \Cref{example: se ci sono abbastanza punti su una retta questa è nel base locus} we have that $\Bb_r(V^d_n)\ne \emptyset$ for all integer $r$ such that $2r\ge d+1$.

 Fix a positive integer $r\le d$ such that $\Bb_r(V^d_n)\ne \emptyset$. Take $A\in \Ss_r(\PP^n)$ with $\nu_d(A)\in \Bb_r(V^d_n)$ and fix $q\in \PP^n\setminus A$ such that $\nu_d(q)\in B(\nu_d(A))$. By \Cref{lemma: dato A in Bb esiste retta contenente un po di pt} there exists a line $L\subset \PP^n$ with $\#(L\cap A)\geq \lceil (d+1)/2\rceil$ from which it follows that $2r\ge d+1$. In particular, we have proved that if $\Bb_r(V^d_n)\ne \emptyset$ then $2r\ge d+1$, and this completes the proof of the if and only if statement.

To prove the three items, we focus on the description of $B(\nu_d(A))$ for some $A\subset \PP^n $ where $\nu_d(A)\in \Bb_r(V^d_n)$ with $r\leq d$.
Assume for the moment that $B(\nu_d(A))\nsubseteq \nu_d(L)$ and take $q'\in \PP^n\setminus (L\cup A)$ with $\nu_d(q')\in B(\nu_d(A))$. 
By \Cref{lemma: dato A in Bb esiste retta contenente un po di pt} there exists a line $R\subset \PP^n$ such that $\# (R\cap A)\geq (d+1)/2$ and $q'\in R$.
Notice that $q'\notin L$ 
so the two lines $L $ and $R$ are distinct and intersect in at most one point. Since $r\le d$, $\#(L\cap A)\ge \lceil (d+1)/2\rceil$ and $\#(R\cap A)\ge \lceil (d+1)/2\rceil$,
we must have $d$ odd, that $L$ and $R$ actually intersect in one point of $A$, 
$r=d$, $A\subset L\cup R$ and $\#(L\cap A) =\#(R\cap A) =(d+1)/2$.

In the particular case $r=d=3$ we notice that there is actually one third line $T\subset \PP^n$ containing two of the three points of $A$ and such that $\# (T\cap L)=\#(T\cap R )=1$. So we have $B(\nu_d(A))=\nu_d((L\cup R\cup T)\setminus A)$.

Otherwise since $ r=d> 3$, no other line of $\PP^n$ contains $(d+1)/2$ points of $A$ and we conclude that $B(\nu_d(A))=\nu_d((L\cup R) \setminus A)$. 

Now we are left with the case $B(\nu_d(A))\subseteq \nu_d(L)$. This happens for $r=d=3$, $r< d$ and $r=d$ with $d $ even and clearly $B(\nu_d(A))=\nu_d(L\setminus A)$.
\end{proof}

\begin{figure}[!ht]
    \centering

\tikzset{every picture/.style={line width=0.75pt}} 

\begin{tikzpicture}[x=0.75pt,y=0.75pt,yscale=-1,xscale=1]

\draw    (556.9,31.23) -- (493.15,100.78) ;
\draw  [fill={rgb, 255:red, 0; green, 0; blue, 0 }  ,fill opacity=1 ] (545.92,40.99) .. controls (545.92,39.97) and (546.74,39.14) .. (547.76,39.14) .. controls (548.78,39.14) and (549.61,39.97) .. (549.61,40.99) .. controls (549.61,42.01) and (548.78,42.83) .. (547.76,42.83) .. controls (546.74,42.83) and (545.92,42.01) .. (545.92,40.99) -- cycle ;
\draw  [fill={rgb, 255:red, 0; green, 0; blue, 0 }  ,fill opacity=1 ] (605.42,85.74) .. controls (605.42,84.72) and (606.24,83.89) .. (607.26,83.89) .. controls (608.28,83.89) and (609.11,84.72) .. (609.11,85.74) .. controls (609.11,86.76) and (608.28,87.58) .. (607.26,87.58) .. controls (606.24,87.58) and (605.42,86.76) .. (605.42,85.74) -- cycle ;
\draw  [fill={rgb, 255:red, 0; green, 0; blue, 0 }  ,fill opacity=1 ] (501.17,89.74) .. controls (501.17,88.72) and (501.99,87.89) .. (503.01,87.89) .. controls (504.03,87.89) and (504.86,88.72) .. (504.86,89.74) .. controls (504.86,90.76) and (504.03,91.58) .. (503.01,91.58) .. controls (501.99,91.58) and (501.17,90.76) .. (501.17,89.74) -- cycle ;
\draw    (538.9,34.23) -- (615.15,91.53) ;
\draw    (615.65,85.53) -- (492.65,90.33) ;
\draw    (256.7,23.2) -- (256.49,54.27) -- (256,125.33) ;
\draw    (456.03,23.2) -- (455.33,125.33) ;
\draw    (346.73,29.17) -- (282.98,98.72) ;
\draw  [fill={rgb, 255:red, 0; green, 0; blue, 0 }  ,fill opacity=1 ] (335.75,38.92) .. controls (335.75,37.91) and (336.58,37.08) .. (337.6,37.08) .. controls (338.62,37.08) and (339.44,37.91) .. (339.44,38.92) .. controls (339.44,39.94) and (338.62,40.77) .. (337.6,40.77) .. controls (336.58,40.77) and (335.75,39.94) .. (335.75,38.92) -- cycle ;
\draw  [fill={rgb, 255:red, 0; green, 0; blue, 0 }  ,fill opacity=1 ] (281.75,39.42) .. controls (281.75,38.41) and (282.58,37.58) .. (283.6,37.58) .. controls (284.62,37.58) and (285.44,38.41) .. (285.44,39.42) .. controls (285.44,40.44) and (284.62,41.27) .. (283.6,41.27) .. controls (282.58,41.27) and (281.75,40.44) .. (281.75,39.42) -- cycle ;
\draw  [fill={rgb, 255:red, 0; green, 0; blue, 0 }  ,fill opacity=1 ] (291,87.67) .. controls (291,86.66) and (291.83,85.83) .. (292.85,85.83) .. controls (293.87,85.83) and (294.69,86.66) .. (294.69,87.67) .. controls (294.69,88.69) and (293.87,89.52) .. (292.85,89.52) .. controls (291.83,89.52) and (291,88.69) .. (291,87.67) -- cycle ;
\draw    (170.9,37.42) -- (107.15,106.97) ;
\draw  [fill={rgb, 255:red, 0; green, 0; blue, 0 }  ,fill opacity=1 ] (159.92,47.17) .. controls (159.92,46.16) and (160.74,45.33) .. (161.76,45.33) .. controls (162.78,45.33) and (163.61,46.16) .. (163.61,47.17) .. controls (163.61,48.19) and (162.78,49.02) .. (161.76,49.02) .. controls (160.74,49.02) and (159.92,48.19) .. (159.92,47.17) -- cycle ;
\draw  [fill={rgb, 255:red, 0; green, 0; blue, 0 }  ,fill opacity=1 ] (219.42,91.92) .. controls (219.42,90.91) and (220.24,90.08) .. (221.26,90.08) .. controls (222.28,90.08) and (223.11,90.91) .. (223.11,91.92) .. controls (223.11,92.94) and (222.28,93.77) .. (221.26,93.77) .. controls (220.24,93.77) and (219.42,92.94) .. (219.42,91.92) -- cycle ;
\draw  [fill={rgb, 255:red, 0; green, 0; blue, 0 }  ,fill opacity=1 ] (115.17,95.92) .. controls (115.17,94.91) and (115.99,94.08) .. (117.01,94.08) .. controls (118.03,94.08) and (118.86,94.91) .. (118.86,95.92) .. controls (118.86,96.94) and (118.03,97.77) .. (117.01,97.77) .. controls (115.99,97.77) and (115.17,96.94) .. (115.17,95.92) -- cycle ;
\draw    (152.9,40.42) -- (229.15,97.72) ;
\draw  [fill={rgb, 255:red, 0; green, 0; blue, 0 }  ,fill opacity=1 ] (319.5,56.42) .. controls (319.5,55.41) and (320.33,54.58) .. (321.35,54.58) .. controls (322.37,54.58) and (323.19,55.41) .. (323.19,56.42) .. controls (323.19,57.44) and (322.37,58.27) .. (321.35,58.27) .. controls (320.33,58.27) and (319.5,57.44) .. (319.5,56.42) -- cycle ;
\draw  [fill={rgb, 255:red, 0; green, 0; blue, 0 }  ,fill opacity=1 ] (299.75,78.42) .. controls (299.75,77.41) and (300.58,76.58) .. (301.6,76.58) .. controls (302.62,76.58) and (303.44,77.41) .. (303.44,78.42) .. controls (303.44,79.44) and (302.62,80.27) .. (301.6,80.27) .. controls (300.58,80.27) and (299.75,79.44) .. (299.75,78.42) -- cycle ;
\draw  [fill={rgb, 255:red, 0; green, 0; blue, 0 }  ,fill opacity=1 ] (330.5,44.92) .. controls (330.5,43.91) and (331.33,43.08) .. (332.35,43.08) .. controls (333.37,43.08) and (334.19,43.91) .. (334.19,44.92) .. controls (334.19,45.94) and (333.37,46.77) .. (332.35,46.77) .. controls (331.33,46.77) and (330.5,45.94) .. (330.5,44.92) -- cycle ;
\draw  [fill={rgb, 255:red, 0; green, 0; blue, 0 }  ,fill opacity=1 ] (273.75,50.92) .. controls (273.75,49.91) and (274.58,49.08) .. (275.6,49.08) .. controls (276.62,49.08) and (277.44,49.91) .. (277.44,50.92) .. controls (277.44,51.94) and (276.62,52.77) .. (275.6,52.77) .. controls (274.58,52.77) and (273.75,51.94) .. (273.75,50.92) -- cycle ;
\draw  [fill={rgb, 255:red, 0; green, 0; blue, 0 }  ,fill opacity=1 ] (284.25,57.67) .. controls (284.25,56.66) and (285.08,55.83) .. (286.1,55.83) .. controls (287.12,55.83) and (287.94,56.66) .. (287.94,57.67) .. controls (287.94,58.69) and (287.12,59.52) .. (286.1,59.52) .. controls (285.08,59.52) and (284.25,58.69) .. (284.25,57.67) -- cycle ;
\draw  [fill={rgb, 255:red, 0; green, 0; blue, 0 }  ,fill opacity=1 ] (149.25,58.42) .. controls (149.25,57.41) and (150.08,56.58) .. (151.1,56.58) .. controls (152.12,56.58) and (152.94,57.41) .. (152.94,58.42) .. controls (152.94,59.44) and (152.12,60.27) .. (151.1,60.27) .. controls (150.08,60.27) and (149.25,59.44) .. (149.25,58.42) -- cycle ;
\draw  [fill={rgb, 255:red, 0; green, 0; blue, 0 }  ,fill opacity=1 ] (170,54.92) .. controls (170,53.91) and (170.83,53.08) .. (171.85,53.08) .. controls (172.87,53.08) and (173.69,53.91) .. (173.69,54.92) .. controls (173.69,55.94) and (172.87,56.77) .. (171.85,56.77) .. controls (170.83,56.77) and (170,55.94) .. (170,54.92) -- cycle ;
\draw  [fill={rgb, 255:red, 0; green, 0; blue, 0 }  ,fill opacity=1 ] (177,59.67) .. controls (177,58.66) and (177.83,57.83) .. (178.85,57.83) .. controls (179.87,57.83) and (180.69,58.66) .. (180.69,59.67) .. controls (180.69,60.69) and (179.87,61.52) .. (178.85,61.52) .. controls (177.83,61.52) and (177,60.69) .. (177,59.67) -- cycle ;
\draw  [fill={rgb, 255:red, 0; green, 0; blue, 0 }  ,fill opacity=1 ] (189.18,69.07) .. controls (189.18,68.05) and (190.01,67.22) .. (191.03,67.22) .. controls (192.04,67.22) and (192.87,68.05) .. (192.87,69.07) .. controls (192.87,70.09) and (192.04,70.91) .. (191.03,70.91) .. controls (190.01,70.91) and (189.18,70.09) .. (189.18,69.07) -- cycle ;
\draw  [fill={rgb, 255:red, 0; green, 0; blue, 0 }  ,fill opacity=1 ] (140.25,68.17) .. controls (140.25,67.16) and (141.08,66.33) .. (142.1,66.33) .. controls (143.12,66.33) and (143.94,67.16) .. (143.94,68.17) .. controls (143.94,69.19) and (143.12,70.02) .. (142.1,70.02) .. controls (141.08,70.02) and (140.25,69.19) .. (140.25,68.17) -- cycle ;
\draw  [fill={rgb, 255:red, 0; green, 0; blue, 0 }  ,fill opacity=1 ] (135.5,73.92) .. controls (135.5,72.91) and (136.33,72.08) .. (137.35,72.08) .. controls (138.37,72.08) and (139.19,72.91) .. (139.19,73.92) .. controls (139.19,74.94) and (138.37,75.77) .. (137.35,75.77) .. controls (136.33,75.77) and (135.5,74.94) .. (135.5,73.92) -- cycle ;

\draw (478.45,80.09) node [anchor=north west][inner sep=0.75pt]    {$R$};
\draw (469.33,26.68) node [anchor=north west][inner sep=0.75pt]    {$\mathbb{P}^{n}$};
\draw (620.2,85.18) node [anchor=north west][inner sep=0.75pt]    {$L$};
\draw (559.95,22.88) node [anchor=north west][inner sep=0.75pt]    {$S$};
\draw (222.67,26.68) node [anchor=north west][inner sep=0.75pt]    {$\mathbb{P}^{n}$};
\draw (427.33,104.52) node [anchor=north west][inner sep=0.75pt]    {$\mathbb{P}^{n}$};
\draw (268.28,78.03) node [anchor=north west][inner sep=0.75pt]    {$R$};
\draw (92.45,86.28) node [anchor=north west][inner sep=0.75pt]    {$R$};
\draw (234.2,91.37) node [anchor=north west][inner sep=0.75pt]    {$L$};
\draw (339.6,40.08) node [anchor=north west][inner sep=0.75pt]   [align=left] {{\small at least $\displaystyle \lceil\frac{d+1}{2}\rceil$ }\\{\small points on a line}};
\draw (22,36.41) node [anchor=north west][inner sep=0.75pt]   [align=left] { {\small $\displaystyle \frac{d+1}{2}$ points of $\displaystyle A$ }\\{\small on each line}};

\end{tikzpicture}

    \caption{Possible configurations of $ A\in\Ss_r(\PP^n)$ such that $\nu_d(A)\in \Bb_r(V^d_n)$ with $ r\leq d$ from \Cref{theorem: Bb per veronesi}.}
    \label{fig:placeholder}
\end{figure}

Now that we characterized emptiness for $\Bb_r(V^d_n)$, we want to do the same for $\Ee_r(V^d_n)$ and $\tilde{\Ee}_r(V^d_n)$, which are our main geometrical interest.  Before proceeding, we make a useful example. 

\begin{example}\label{ex: base induzione d=3 Ee vuoto} 
Let $d=r=3$ and let $A\subset \PP^n$ be a set of three general points. We want to show that $ \nu_3(A)\notin \Ee_3(V^3_n)$. 
By \Cref{theorem: Bb per veronesi} we know that   $B(\nu_3(A))=\nu_3(L\cup R\cup T)\setminus \nu_3(A)$ where $L,R,T\subset \PP^n$ are three distinct lines, each containing two of the three points of $A$ and such that any two of the lines intersect in a unique point of $A$. Hence $\nu_3(L\cup R\cup T)$ is singular only at $\nu_3(A)$, so $\nu_3(A)\notin \Ee_3(V^3_n)$. 
\end{example}

Next we exhibit a configuration of points that actually give $\tilde{\Ee}_r(V^d_n)\neq \emptyset$ in the particular case $d=r$.
\begin{example}\label{ex: buono retta doppia con d punti è tangential contact locus}
Fix $n,d\geq 2$. Let $L\subset \PP^n$ be a line and let $A\subset L$ with $\#A=d$. We want to show that for any $p\in L\setminus A$ we have that $\langle 2A\cup 2p \rangle=\langle 2A\rangle$, hence giving $\nu_d(A)\in \Ee_d(V^d_n)$. For this, notice that it is enough to prove that for any length 2 zero-dimensional scheme $v$ supported at $p$ we have $\langle 2A \rangle=\langle 2A\cup v\rangle$, or equivalently, that $v$ is in the base locus of $\Ii_{2A}(d)$. Since $\#A=d$ and $\deg(2A\cap L)=2d\geq d+1$, \Cref{example: se ci sono abbastanza punti su una retta questa è nel base locus} gives that $L$ is contained in the base locus of $\Ii_{2A}(d)$. So if $v\subset  L$ then there is nothing to prove. Assume now that $v \not\subset L$ so that $ \langle v\cup L \rangle=M\cong \PP^2$. Since $v\subset M$, we can restrict to work on the plane $M$. By contradiction we assume that $\langle 2A \rangle\neq\langle 2A\cup v\rangle$, hence we have that $h^0(\Ii_{(2A\cup v)\cap M,\PP^n}(d))<h^0(\Ii_{(2A)\cap M,\PP^n}(d)) $. 
Restricting to the plane, we have that $\deg((2A\cup v)\cap L)=2d+1$ which implies that $2L$ is in the base locus of $\Ii_{(2A\cup v)\cap M,M}(d)$, hence $H^{0}(\Ii_{(2A\cup v)\cap M,M }(d))=H^{0}(\Ii_{(2A)\cap M,M }(d))$ which contradicts the assumption just made since $M=\PP^2$ is arithmetically Cohen-Macaulay (so $H^{0}(\Ii_{(2A\cup v)\cap M,M }(d))=H^{0}(\Ii_{(2A\cup v)\cap M,\PP^n }(d))$ and $
H^{0}(\Ii_{(2A)\cap M,M }(d))=
H^{0}(\Ii_{(2A)\cap M,\PP^n }(d))$).
\end{example}

We are ready to characterize emptiness of $\Ee_r(V^d_n)$ and $\tilde{\Ee}_r(V^d_n)$.

\begin{theorem}\label{theorem: Ee e tilde Ee per veronesi}
Fix integers $n\ge 2$ and $d\ge 2$. Then
$$
\Ee_r(V^d_n) \text{ and } \tilde{\Ee}_r(V^d_n) \text{ are nonempty} \text{ if and only if } r\geq d. 
$$
\end{theorem}
\begin{proof}
 For the "if" part we observe that \Cref{ex: buono retta doppia con d punti è tangential contact locus} gives a configuration of $r=d$ points in $\tilde{\Ee}_d(V^d_n)\subseteq \Ee_d(V^d_n)$. Hence, by \Cref{remark: una volta in Ee se aggiungo punti resto in Ee} we have that $ \Ee_r(V^d_n)\supseteq \tilde{\Ee}_r(V^d_n)\neq \emptyset$  for all $r\geq d$. So let us focus on the "only if" part of the statement.

 By \Cref{theorem: Bb per veronesi}, if $2r< d+1$ then $\Bb_r(V_n^d)$ is empty, hence the emptiness of $ \Ee_r(V^d_n)$ should only be proved in the range $\lceil (d+1)/2 \rceil\leq r<d$, which implies that $d\geq 3$. Assume that the statement is false and so that there exists $A\subset \PP^n$ with $\#A=r<d$ such that $\nu_d(A)\in \Ee_r(V_n^d)$ and let $\nu_d(q)\in E(\nu_d(A))$. By \Cref{lemma: dato A in Bb esiste retta contenente un po di pt}  there is a line $L\subset \PP^n$ such that $q\in L\setminus (L\cap A)$  and $e:= \#(A\cap L)\ge \lceil(d+1)/2\rceil$. 
 
 Assume first $A\subset L$. If $r\leq d-2$ let $G\subset \PP^n$ be a hypersurface of degree $d-1-r$ not passing through $A\cup \{ q\}$, and if $r=d-1$ then we set $G=\emptyset$. Let $M$ be a general hyperplane containing $L$, so in particular if $n=2$ then $M=L$. Let $K$ be the union of $r$ general hyperplanes, each of them containing a different point of $A$. On the one hand we have $A\subset \mathrm{Sing}(K\cup M\cup G)$ because $A\subset K,M$. On the other hand, since $q\notin K$ we have $q\notin  \mathrm{Sing}(K\cup M\cup G)$. But this contradicts the assumption $\nu_d(q)\in E(\nu_d(A))$, or equivalently, we found that $H^0(\Ii_{2A\cup 2q}(d))\neq H^0(\Ii_{2A}(d))$, which is impossible. 
 
Hence $A\not\subset L$. Since $A\not\subset L$ and $\lceil (d+1)/2 \rceil\leq e <r$ this in particular implies that $d\geq 6$. Let $H$ be a general hyperplane containing $L$, so in particular $H=L$ if $n=2$. Consider the exact sequence of $2H$ applied to $2(A\cup p)$:
\begin{equation*}
0 \to \Ii_{2(A\setminus A\cap L)}(d-2)\to \Ii_{2(A\cup p)}(d)\to \Ii_{2p\cup 2(A\cap L),2H}(d)\to 0,
\end{equation*}
where the residual is given by the double points not supported in $L$.
 We have $\#(A\setminus A\cap L) <d-e\le \lfloor (d-1)/2\rfloor$ so $h^1(\Ii_{2(A\setminus A\cap L)}(d-2))=0$ (see for instance \cite[Theorem 1.2]{laface2024ample}).
 Hence $h^1(\Ii_{2(A\cup p)}(d))=h^1(\Ii_{2p\cup 2(A\cap L)}(d))$ and from the previous case in which we considered aligned points, it would be impossible to have $h^1(\Ii_{2(A\cup p)}(d))=n+1+h^1(\Ii_{2(A\cap L)}(d))$ since $e<d$.
\end{proof}

When $r=d$ we can go a step further and precisely characterize both $\Ee_d(V^d_n)$ and $\tilde{\Ee}_d(V^d_n)$.

\begin{theorem}\label{theorem: r=d veronesi Ee}
Let $d\geq 2,n\geq 2$ and let $A\subset \PP^n$ with $\#A=d$. Then 
$$
\tilde{\Ee}_d(V^d_n)=\Ee_d(V^d_n)=\{\nu_d(A)\,|\, A\in \Ss_d(\PP^n) \text{ is contained in a line} \}.
$$
Moreover, the contact locus of any $\nu_d(A)\in \Ee_d(V^d_n)$ is a degree-$d$ rational normal curve $\nu_d(L)\subset V^d_n$ where $L\subset \PP^n$ is a line containing $A$.
\end{theorem}

\begin{proof}
By \Cref{ex: buono retta doppia con d punti è tangential contact locus}, to conclude the proof we only have to show that any set $A\subset \PP^n$ that is not aligned cannot give a configuration for $\Ee_d(V^d_n)$. We recall that since $\Ee_d(V^d_n)\subset \Bb_d(V^d_n)$, if we fix $A\subset \PP^n$ with $\nu_d(A)\in \Ee_d(V_d^n)$ then by \Cref{lemma: dato A in Bb esiste retta contenente un po di pt} there is a line $L\subset \PP^n$ such that $q\in L\setminus (L\cap A)$  and $e:= \#(A\cap L)\ge \lceil(d+1)/2\rceil$. The case $d=3$ is settled in \Cref{ex: base induzione d=3 Ee vuoto}, so we can assume $d\geq 4$.
Let $H$ be a general hyperplane containing $L$, so in particular $H=L$ if $n=2$. Consider the exact sequence of $2H$ applied to $2(A\cup p)$:
\begin{equation*}
0 \to \Ii_{2(A\setminus A\cap L)}(d-2)\to \Ii_{2(A\cup p)}(d)\to \Ii_{2p\cup 2(A\cap L),2H}(d)\to 0.
\end{equation*}
 We have $\#(A\setminus A\cap L) =d-e\le \lfloor (d-1)/2\rfloor$ so $h^1(\Ii_{2(A\setminus A\cap L)}(d-2))=0$ (see for instance \cite[Theorem 1.2]{laface2024ample}). Hence $h^1(\Ii_{2(A\cup p)}(d))=h^1(\Ii_{2p\cup 2(A\cap L)}(d))$ and the latter cannot be equal to $n+1+h^1(\Ii_{2(A\cap L)}(d))$ since $e<d$.
\end{proof}

We conclude by revisiting \Cref{connection terracini} to discuss easy examples of interaction between $\Bb_r(V^d_n),\Ee_r(V^d_n)$ and $\TT_r(V^d_n)$. Combining \Cref{theorem: Bb per veronesi}, \Cref{theorem: Ee e tilde Ee per veronesi} and \cite[Theorem 6.6]{GSTT}, for all $n\geq 2$ we have
$$\Bb_3(V^3_n)=\Ss_3(V^3_n) \supset \tilde{\Ee}_3(V^3_n)=\Ee_3(V^3_n)=\TT_3(V^3_n)= \{ \nu_3(A) \,|\, A\in \Ss_3(\PP^n) \text{ is contained in a line} \}.$$ Moreover, for all $n\ge2,d\ge4$ we have
$$\tilde\Ee_d(V^d_n)=\Ee_d(V^d_n)=\{ \nu_d(A) \, |\, A\in \Ss_d(\PP^n) \text{ is contained in a line }\}\subsetneq \TT_d(V^n_d),$$
and by \Cref{remark: una volta in Ee se aggiungo punti resto in Ee} for all $r\geq d+1$ there is always a family of $\Ee_r(V^n_d)$ given by at least $d$ collinear points contained in the corresponding Terracini locus. 

Lastly, we remark that by \cite[Theorem 1.2]{COV17} the cases $(n,d,r)\in \{ (2,6,9),(3,4,8),(5,3,9)\} $ correspond to (the only) tangentially weakly defective Veronese varieties and for all such $(n,d,r)$ we have that the general $T\in \sigma_r(V^d_n)$ lies in exactly two secant planes given by two general $A,B\in \Ss_r(X)$. Hence the tangent space {at $T$} of the corresponding secant variety {is such that} $T_T(\sigma_r(V^d_n))=\langle 2A\rangle=\langle 2B\rangle$ with dimension $r(n+1)-1<N$. So the general $A,B\in \Ss_r(X)$ are such that both $A,B\in \tilde{\Ee}_r(V^d_n) $ while $A,B\notin \TT_r(V^d_n)$ and in particular $\tilde{\Ee}_r(V^d_n)\supsetneq \TT_r(V^d_n)$.

\section{Segre-Veronese  varieties}\label{section: segre-veronese}
In this section we focus on Segre-Veronese varieties. Fix positive integers $k$, $n_1,\dots,n_k$, $d_1,\dots,d_k$. Set $Y:= \PP^{n_1}\times \cdots \times \PP^{n_k}$, $N=\binom{n_1+d_1}{d_1}\cdots \binom{n_k+d_k}{d_k}-1$ and let $\nu:Y\rightarrow \PP^N$ be the Segre-Veronese embedding of $Y$ via the complete linear system given by $\Oo_Y(d_1,\dots,d_k)$. We will denote by $X=\nu(Y)$ the Segre-Veronese variety embedded with degrees $d_1,\dots,d_k$.

We will further use the following notation.
\begin{notation}\label{notation SV}
Let $k\geq 2$. For each $i\in \{1,\dots ,k\}$ let
\begin{itemize}
    \item $Y_i=\PP^{n_1}\times \cdots \times \PP^{n_{i-1}}\times \PP^{n_{i+1}}\times \cdots \times \PP^{n_k}$ be the space in which we forget the $i$th factor of $Y$;
    \item $\pi_i: Y \to \PP^{n_i}$ be the projection on the $i$-th factor of $Y$;
    \item $\eta_i: Y\rightarrow Y_i$ be the projection on $Y_i$;
    \item $\epsilon _i\in \NN^k$ be $\varepsilon_i=(0,\dots,0,1,0,\dots,0)$  with the 1 in the $i$th position.
\end{itemize}
\end{notation}
For each irreducible curve $C\subset Y$ the multidegree of $C$ is the element $(b_1,\dots,b_k)\in \NN^k$ such that $b_i=\deg(\Oo_C(\epsilon _i))$. In other words, if $n_i=1$ then $b_i=\deg(\pi_{i|C})$, while if $n_i\geq 2$ then $b_i=\deg(\pi_i(C))\cdot\deg(\pi_{i|C})$. Moreover, notice that $C$ has multidegree $(b_1,\dots,b_k)$ with $b_i=0$ if and only if $\pi_i(C)$ is a point. For each reduced curve $C\subset Y$ the multidegree of $C$ is the sum of the multidegrees of the irreducible components of $C$. Notice also that the notion of multidegree of a curve coincide with the usual notion of degree in the case $k=1$.

\begin{remark}\label{remark: se i punti sono distinti quando rimuovo un fattore e avevo h1 positivo resto con h1 pos}
Let $k\geq 2$ and let $Z\subset Y=(\PP^1)^{\times k}$ be a zero-dimensional scheme such that $h^1(\Ii_Z(t,\dots ,t))>0$ and $\pi_{1|Z}$ is an embedding. So in particular $\eta_{k|Z}$ is an embedding and $\deg(\eta_k(Z)) =\deg (Z)$. Recall that $Y_k\cong (\PP^1)^{\times (k-1)}$ in which we omit the last factor of $Y$ and that $\eta_k: Y\to Y_k$ is the projection. Then,
$$h^1(Y_k,\Ii_{\eta_k(Z)}(t,\dots ,t)) >0.$$
Indeed, let $H\in |\Oo_Y(0,\dots,0,t)|$ be the preimage of a finite $S\subset \PP^1$ with $\#S\leq t$ and $S\cap \pi_k(Z)=\emptyset$. The residual exact sequence of $H$ gives $h^1(\Ii_{Z\cap H,H }(t,\dots,t))=0$ since $Z\cap H=\emptyset$, so 
$$h^1(\Ii_{\Res_H(Z)}(t,\dots,t,0))\ge h^1(\Ii_{Z}(t,\dots,t))>0.$$ 
But now $\Res_H(Z)=Z$ and $\deg(\eta_k(Z))=\deg(Z)$. Moreover, $H^0(\Oo_Y(t,\dots,t,0))\cong H^0(\Oo_{Y_k}(t,\dots,t))$, which implies $h^1(\Ii_{\Res_H(Z)}(t,\dots,t,0))=h^1(Y_k,\Ii_{\eta_k(Z)}(t,\dots,t))$.
\end{remark}

We want to start with a technical lemma that generalizes \Cref{lemma: famoso} for all Segre-Veronese varieties. For this purpose, let us first see an example on how the lemma generalizes in the multiprojective space when the corresponding Segre-Veronese has {at least} one of its degrees equal to one and the degree of the zero-dimensional scheme is at most three.

\begin{example}\label{example: famoso lemma nel multiproj with min deg=1}
Let $ Y=\PP^{n_1}\times \cdots \times \PP^{n_k} $ be embedded in multidegree $d_1,\dots,d_k$ where $d_i=1$ for some $i\in \{ 1,\dots,k\}$. Let $Z\subset Y$ be a zero-dimensional scheme with $\deg(Z)\leq 3$ and such that $h^1(\Ii_Z(d_1,\dots,d_k))>0$. First notice that since $h^1(\Ii_Z(d_1,\dots,d_k))>0$ we actually have that $\deg(Z)=3$ and $\langle \nu(Z) \rangle$ is contained in a line, by \Cref{prel01}. Moreover, since $\nu(Y)$ is scheme-theoretically cut out by quadrics, such a line containing $\langle \nu(Z)\rangle $ is completely included in $\nu(Y)$ 
and it is the image of a curve of multidegree $\varepsilon_i$ where $i$ must correspond to an index for which $d_i=1$ by the structure of the rulings of $\nu(Y)$.
\end{example}

The following is a generalization of \Cref{lemma: famoso}.
\begin{lemma}\label{lemma: famoso per SV}
Fix an integer $k\ge 1$  and positive integers $n_i$, $d_i$, $1\le i\le k$. Let $Y= \PP^{n_1}\times \cdots \times \PP^{n_k}$ and $t:= \min\{d_1,\dots,d_k\}$.
Let $Z\subset Y$ be a zero-dimensional scheme such that $\deg(Z)\le 2t+1$ and $h^1(\Ii_Z(d_1,\dots ,d_k)) >0$. Then there is $i\in \{1,\dots ,k\}$ such that
$d_i=t$ and a curve $L\subset Y$ of multidegree $\epsilon_i$ such that $\deg(L\cap Z)\ge t+2$. 
\end{lemma}

Before proving the complete result, let us treat separately the case where $Y=(\PP^1)^{\times k}$ is embedded in multidegree $ d_1=\dots=d_k=t\geq 1$  and the projection of $Z$ in each of the factors is an embedding. 

\begin{example}\label{example: lemma famoso tutti P1 e gradi uguali}
With assumptions as in \Cref{lemma: famoso per SV}, we prove the result for the particular case of
$$
n_1=\dots=n_k=1,\;d_1=\cdots=d_k=t\geq 1,
$$
and $Z\subset Y$ being such that $\pi_{i|Z}$ is an embedding for all $i\in \{1,\dots,k \}$. Notice in particular that $Z$ is curvilinear. 

We work by induction on $k\geq 1$.
The base case $k=1$ is \Cref{lemma: famoso}, so assume $k\geq 2$.  Since $\pi_{1|Z}$ is an embedding, $\eta_{k|Z}: Z\to Y_k$ is an embedding. Hence $\deg(\eta_k(Z)) =\deg (Z)\le 2t+1$. Recall that $Y_k\cong (\PP^1)^{k-1}$ in which we omit the last factor of $Y$ and that $\eta_k: Y\to Y_k$ is the projection.

By \Cref{remark: se i punti sono distinti quando rimuovo un fattore e avevo h1 positivo resto con h1 pos} we have $h^1(Y_k,\Ii_{\eta_k(Z)}(t,\dots ,t)) >0$. By the inductive assumption on $k$ there is an integer $i\in \{1,\dots ,k-1\}$ and a curve $R\subset Y_k$ of multidegree $\epsilon_i$ such that $\deg(R\cap \eta_k(Z))\ge t+2$. Set $M:= \eta_k^{-1}(R)$. We have $M\cong \PP^1\times \PP^1$ and $M$ is the complete intersection of $k-2$ hypersurfaces $H_j\in |\Oo_Y(\epsilon_j)|$, for all $j\in \{1,\dots ,k-1\}\setminus \{i\}$. Fix one of these hypersurfaces $H_j$ and consider the residual sequence of $H_j$. 

Clearly,
$\deg(Z\cap H_j)\le \deg (Z)\le 2t+1$.
If $h^1(H_j,\Ii_{Z\cap H_j,H_j}(t,\dots ,t)) >0$, we conclude by the inductive assumption on $k$. Indeed there is a curve $L$ of multidegree $\epsilon_s$ such that $\deg(L\cap Z)\ge \deg(L\cap (Z\cap H_j))\ge t+2$.

Hence, we are left with the case in which $h^1(H_j,\Ii_{Z\cap H_j,H_j}(t,\dots ,t)) =0$. In this case we have $h^1(\Ii_{\Res_{H_j}(Z)}(t,\dots ,t,t-1,t,\dots,t)) >0$. Now since $H_j\supseteq M$, we have that $\deg(H_j\cap Z)\ge \deg(\eta_k(Z)\cap R)\ge t+2$. 
Hence since by assumption $\deg(Z)\leq 2t+1$, we have $\deg(\Res_{H_j}(Z))\le t-1\le 2(t-1)+1$. 

We work now by induction on $t\geq 1$, where the base case $t=1$ is \Cref{example: famoso lemma nel multiproj with min deg=1}.
 We apply  the inductive hypothesis to $\Res_{H_j}(Z)$ and we have the existence of a curve $S$ such that $\deg(S\cap \Res
 _{H_j}(Z))\ge(t-1)+2=t+1$
  and we get a contradiction.
\end{example}

Now we are ready to prove \Cref{lemma: famoso per SV}.
\begin{proof}[Proof of \Cref{lemma: famoso per SV}] 
We will use induction on the integer $t\geq 1$ where the base case $t=1$ is \Cref{example: famoso lemma nel multiproj with min deg=1}. Moreover, we will use induction on the positive integer $n:= \dim Y=n_1+\cdots +n_k$. The base case $n=1$, and so $k=1$, is covered by \Cref{lemma: famoso}. 

We first prove the result under the convenient assumption that all degrees coincide, namely $$d_1=\cdots =d_k=t,$$ and at the end we will show how we can always reduce to consider all degrees equal. If $Y=(\PP^1)^{\times k}$ for some $k\geq 1$ and each $\pi_{i|Z}:Z\rightarrow \PP^1$ is an embedding, the lemma is true by \Cref{example: lemma famoso tutti P1 e gradi uguali}. In the following, we either prove the result by induction or we show how one can always reduce to \Cref{example: lemma famoso tutti P1 e gradi uguali}.

Let $H\in |\Oo_Y(\varepsilon_i)|$ be such that $\deg( Z\cap H)$ is maximal for some $i\in \{1,\dots,k \}$. Without loss of generality and by permuting the factors of $Y$ if necessary we take $i=k$. Notice that $\deg(Z\cap H)\geq \min \{\deg(Z),n_k \}$ since the dimension of the $k$th factor is $n_k$. Consider the residual sequence for $H$:
$$
0\rightarrow \Ii_{\Res_H(Z)}(t,\dots,t,t-1)\rightarrow \Ii_Z(t,\dots,t)\rightarrow \Ii_{Z\cap H,H}(t,\dots,t) \rightarrow 0.
$$
We look at $\deg(\Res_H(Z))=\deg(Z)-\deg(Z\cap H)$. 
\medskip

First, assume $\deg(\Res_H(Z))\le 2(t-1)+1=2t-1$. By inductive assumption on $t$ there exists a curve $R\subset Y$ of multidegree $\varepsilon_k$ such that $\deg(\Res_H(Z) \cap R)\geq (t-1)+2=t+1$. If $\deg(\Res_H(Z) \cap R)\geq t+2$ then it is sufficient to take $L:=R$ to get the result. Hence we assume $ \deg(\Res_H(Z) \cap R)=t+1$. Since the multidegree of $R$ is $\varepsilon_k$ we have that $\pi_i(R)$ is a single point for all $i<k$, so in particular also $\pi_1(R)$ is a point and $|\Ii_R(\varepsilon_1)|\neq \emptyset$. Take a divisor $M\in |\Ii_R(\varepsilon_1)| $ passing through $R$ and consider the residual sequence for $M$:
    $$
0\rightarrow \Ii_{\Res_M(Z)}(t-1,t,\dots,t)\rightarrow \Ii_Z(t,\dots,t)\rightarrow \Ii_{Z\cap M,M}(t,\dots,t) \rightarrow 0.
$$
If $h^1(M,\Ii_{M\cap Z,M}(t,\dots,t))>0$ then we can conclude by induction on the integer $n $. Hence, we assume $h^1(M,\Ii_{M\cap Z,M}(t,\dots,t))=0$, which implies $h^1(\Ii_{\Res_M(Z)}(t-1,t, \dots,t))>0$. But now we have that $\deg(\Res_M(Z))\leq \deg(Z)-(t+1)\leq t $ and by inductive assumption on $t$ we would have the existence of a line $Q$ of multidegree $\varepsilon_1$ such that $ \deg(\Res_M(Z)\cap Q)\geq (t-1)+2=t+1$, which is impossible.
\smallskip

Otherwise, we have $\deg(\Res_H(Z))\geq  2t$. By assumption we know that $\deg(Z)\leq 2t+1$, hence $\deg(Z\cap H)\leq 1$ and since we chose $H$ such that $\deg(Z\cap H)$ is maximal, this means that actually $\deg(Z)=2t+1$, $\deg(Z\cap H)=1\geq \min \{\deg(Z),n_k \}=n_k$, hence the $k$th factor is a $\PP^1$. \\
By applying the analogous argument above to other divisors of $|\Oo_Y(\varepsilon_i)|$ with $i<k$ taken such that their intersection with $Z$ is maximal, we either conclude by induction or we get that all $n_i=1$, so $Y=(\PP^1)^{\times k}$ and each $\pi_{i|Z}: Z\to \PP^1$ is an embedding. Hence, the lemma is true by \Cref{example: lemma famoso tutti P1 e gradi uguali}.

To conclude the proof, we have to analyze the remaining case when not all $d_i$'s coincide. 
Without loss of generality and by relabeling if necessary, we may assume $t=d_1\leq \cdots \leq d_k$, with $d_k>t$, otherwise there is nothing to prove. We either prove the result by induction or we prove that we can always reduce to the case that all degrees are equal. 

Take $H\in |\Oo_Y(\varepsilon_k)|$ such that $\deg(Z\cap H )$ is maximal, in particular, we have $\deg(Z\cap H)\geq \min\{\deg(Z),n_k \}$. Consider the residual exact sequence of $H$:
$$
0\rightarrow \Ii_{\Res_H(Z)}(d_1,\dots,d_{k-1},d_k-1)\rightarrow \Ii_Z(d_1,\dots,d_k)\rightarrow \Ii_{Z\cap H,H}(d_1,\dots,d_k)\rightarrow 0.
$$
If $h^1(H,\Ii_{Z\cap H,H}(d_1,\dots,d_k))>0$, then it is sufficient to use the inductive assumption on the dimension $n$ of $Y$. Otherwise $h^1(H,\Ii_{Z\cap H,H}(d_1,\dots,d_k))=0$, but then we necessarily have $h^1(\Ii_{\Res_H(Z)}(d_1,\dots ,d_k-1))>0$ and  we look at $\Res_H(Z) $. Since $d_k>t$ we have that $d_k-1\geq t$. In particular, if $d_{k}-1=t$ we reorder the factors of $Y$ such that the corresponding degrees are in an increasing order and the new degree in the $k$th factor is strictly greater than $t$. We start all over the procedure of taking a divisor $H'\in |\Oo_Y(\varepsilon_k)|$ whose intersection with $\Res_H(Z):=Z'$ is maximal and applying the corresponding residual sequence again either we find a contradiction or we have $h^1(\Ii_{\Res_{H'}(Z')}(d_1,\dots,d_k-1))>0$. We keep going like this until either we find a contradiction (possibly because there is nothing more to consider in the corresponding residue) or we get that all degrees coincide and equal $t$. 
\end{proof}

 The following is the analogue of \Cref{lemma: dato A in Bb esiste retta contenente un po di pt} for Segre-Veronese varieties.
\begin{lemma}\label{lemma: dato A in Bb esiste retta contenente un po di pt SV}
Let $t=\min_i\{d_i \}$,   $r\le t$ and let $A\in \Ss_r( Y)$. Assume $\nu(A)\in \Bb_r(X)$ and fix $q\in Y$ such that $\nu(q)\in B(\nu(A))$. Then there is $i\in \{1,\dots ,k\}$ such that $d_i=t$ and a curve $L\subset Y$ of multidegree $\epsilon_i$ such that $ \#(A\cap L)\geq \lceil(t+1)/2\rceil$ and $q\in L$. 
\end{lemma}

\begin{proof}   
Take $A\subset Y$ with $\nu(A)\in \Bb_r(X)$ and fix $q\in Y$ such that $\nu(q)\in B(\nu(A))$. Let $Z\subset Y$ be a zero-dimensional scheme such that $\nu(Z)$ is a critical scheme for the pair $(\nu(A),\nu(q))$. Notice that such a scheme always exists by \Cref{lemma: existence critical scheme characterization}. Hence $\deg(Z)\leq 2r+1\leq 2t+1$ and $h^1(\Ii_Z(d_1,\dots,d_k))>0$. By \Cref{lemma: famoso per SV} there is a curve $L\subset Y$ of multidegree $\varepsilon_i$ for some $i\in \{1,\dots,k\}$ with $d_i=t$ such that $\deg(L\cap Z)\geq t+2$. Hence $\#(A\cap L)\geq \lceil (t+1)/2\rceil $. We now want to prove that $Z\subset L$ because this would give $q\in L$.

If $t=1$ then $r=1$ and on one hand we have $\deg(Z)\leq 2t+1=3$ while on the other hand we just proved $\deg(Z\cap L)\geq t+2=3$. So assume $t\geq 2$. Assume by contradiction that $Z\not\subseteq L$. Since $L$ is scheme-theoretically cut out by hypersurfaces of degree $\varepsilon_j$ for $j\in \{1,\dots,k \}$, there is $j\in \{ 1,\dots,k\}$ such that $Z\cap H_j\neq Z$. Consider the residual exact sequence of $H_j$:
$$
0\rightarrow \Ii_{\Res_{H_j}(Z)}(d_1\dots,d_{j-1},d_j-1,d_{j+1}\dots,d_k)\rightarrow \Ii_{Z}(d_1,\dots,d_k)\rightarrow \Ii_{H_j\cap Z,H_j}(d_1,\dots,d_k).
$$Since $\deg(Z\cap H_j)\geq \deg(Z\cap L)\geq t+2$ we have that $\deg(\Res_{H_j}(Z))=\deg(Z)-\deg(Z\cap H_j)\leq t-1$ and we conclude that $h^1(\Ii_{\Res_{H_j}(Z)}(d_1\dots,d_{j-1},d_j-1,d_{j+1}\dots,d_k))=0$, otherwise we would have a contradiction with \Cref{lemma: famoso per SV}. Hence, $h^1(\Ii_{Z\cap H_j,H_j}(d_1,\dots,d_k))>0$. Now since $H_j$ is arithmetically Cohen-Macaulay, we have $h^1(\Ii_{Z\cap H_j}(d_1,\dots,d_k))=h^1(H_j,\Ii_{Z\cap H_j,H_j}(d_1,\dots,d_k))>0$. By \Cref{lemma: properties critical scheme} (cf. also \Cref{remark: riformulazione proprietà schema critico in cohomologia per casi belli}) we know that $h^1(\Ii_{\tilde{Z}})(d_1,\dots,d_k)=0$ for any $\tilde{Z}\subsetneq Z $. Hence we get $Z\cap H_j=Z$ which is impossible since we assumed $Z\not\subseteq L$.
\end{proof}

We are now ready to characterize emptiness of base loci for any Segre-Veronese variety, including Segre. Moreover we give a complete description of the base locus $ B(A)\subset X\setminus A$ associated to a given $A\in \Bb_r(X) $ for small $r$.
\begin{theorem}\label{theorem: Bb per SV}
Fix $k\ge2$ and positive integers $n_i$, $d_i$, for $1\le i\le k$. 
Set $Y:= \PP^{n_1}\times \cdots \times \PP^{n_k}$, $X:= \nu(Y)$ and $t:= \min\{d_1,\dots,d_k\}$. 
Then $$\Bb_r(X) \ne \emptyset \mbox{ if and only if }2r\ge t+1.$$ 
 Assume $\lceil (t+1)/2\rceil \leq r \leq t$ and let $A\in \Ss_r( Y)$ be such that $\nu(A)\in \Bb_r(X)$. 
 \begin{itemize}
     \item  If $r=t=1$, then $B(\nu(A))\cup \nu(A)$ is the Segre-Veronese image of the union of all rulings of $Y$ with $d_i=t$. 
     \item If $t\ge 5$ and $r=t$ odd then either there exists a curve $L$ of multidegree $\epsilon_i$, with $d_i=t$, containing at least $(t+1)/2$ points of $A$ and such that $B(\nu(A))=\nu(L\setminus A)$, or there is 
     another curve $R$ of multidegree $\epsilon_j$ with $d_j=t$, such that $R$ and $L$ intersect in a point of $A$ and both contain $(t+1)/2$ points of $A$.
     We have $B(\nu(A)) =\nu((L\cup R)\setminus A)$ and moreover if $i=j$ then $n_i\geq 2$.
    \item If either $r<t$ or $r=t$ is even, then 
    there exists a curve $L$ of multidegree $\epsilon_i$, with $d_i=t$, containing at least $\lceil(t+1)/2\rceil$ points of $A$ and such that $B(\nu(A))=\nu(L\setminus A)$. 
    \item If $r=t=3$ then 
    \begin{itemize}
        \item either there is a curve $L$ of multidegree $\varepsilon_i$ with $d_i=t$ containing at least two points of $A$ and such that $B(\nu(A))=\nu(L\setminus A)$,
        \item or there are two curves $L,R$ of multidegree $\varepsilon_i,\varepsilon_j$ with $i\neq j$ such that $d_i=d_j=t$, each of them containing two of the three points of $A$, and intersecting in one point of $A$, 
        \item  or 
    there is an index $i$ such that $n_i\geq 2$ and $d_i=t$
    and three curves $L,R,S\subset Y$ of multidegree $\varepsilon_i$, passing each of them through two points of $A$. 
    In this case $B(\nu(A))=\nu((L\cup R\cup S)\setminus A)$.
    \end{itemize}
    
 \end{itemize}
\end{theorem}
\begin{proof}
Assume $\Bb_r(X)\neq \emptyset$, let $A\subset Y$ be of cardinality $r$ with $\nu(A)\in \Bb_r(X)$ and fix $q\in Y\setminus A$ with $\nu(q)\in B(\nu(A))$. By \Cref{lemma: dato A in Bb esiste retta contenente un po di pt SV} there exists $i\in \{1,\dots ,k\}$ such that $d_i=t$ and a curve $L\subset Y$ of multidegree $\epsilon_i$ such that $ \#(A\cap L)\geq \lceil(t+1)/2\rceil$, from which it follows that $2r\geq t+1$. On the other hand, let $A\subset Y$ be a finite set such that $\#(A\cap L)\geq \lceil (t+1)/2\rceil $ for some line $L\subset Y$ of multidegree $\varepsilon_i$ where $i$ is such that $d_i=t$. Since $\deg(2A\cap L)=\#2(A\cap L)\geq t+1$ we have that $L$ is contained in the base locus of $\Ii_{2A}(d_1,\dots,d_k)$ so $\Bb_r(X)\neq \emptyset $ for all integer $r$ such that $2r\geq t+1$.
\medskip

Now we are left to prove the items of the statement. For this, we focus on the description of $B(\nu(A))$ for some $A\subset Y $ such that $\nu(A)\in \Bb_r(X)$ with $r\leq t$. Take for the moment $t=1$ so that $r=1$. To understand $B(\nu(A))$ we have to look at the intersection of $\nu(Y)$ with the tangent space $T_{\nu(A)}\nu(Y)$. These are exactly the Segre-Veronese image of all points lying on a ruling of $Y$ corresponding to $d_i=t$.

So from now we can consider $t\ge 2$. Since $2r\ge t+1$ we have that $r\ge 2$. Assume for the moment that $B(\nu(A))\not\subseteq \nu(L)$ and note that $L$ is uniquely determined by $2$ of its points. Take $q'\in Y\setminus (L\cup A)$ with $\nu(q')\in B(\nu(A))$. By \Cref{lemma: dato A in Bb esiste retta contenente un po di pt SV} there is $j\in \{1,\dots ,k\}$ such that $d_j=t$ and a curve $R\subset Y$ of multidegree $\epsilon_j$ such that $q'\in R$ and $\#(A\cap L)\ge \lceil (t+1)/2\rceil$. Notice that $q'\notin L$ so the two lines $L,R$ are distinct and they intersect in at most one point. Since $r\leq t$, $\#(A\cap L)\geq \lceil(t+1)/2\rceil$, $\#(A\cap R)\geq \lceil (t+1)/2\rceil$ we must have $t$ odd, that the intersection between $L$ and $R$ is one point of $A$, $r=t$, $A\subset L\cup R$ and $\#(A\cap L)=\#(A\cap R)=(t+1)/2 $. When this configuration occurs, if $ i=j$ then this forces the dimension of the $i$-th factor of $Y$ to be $n_i\geq 2$. Moreover, in the particular case $r=t=3$, if $i=j$ and $n_i\geq 2$ then there is actually a third curve $T\subset Y$ of multidegree $\varepsilon_i$ containing two of the three points of $A$ and such that $\#(T\cap L)=\#(T\cap R)=\#(R\cap L)=1$. So $B(\nu(A))=\nu(L\cup R\cup T)\setminus \nu(A)$. Otherwise since $r=t>3$ then no other line contains $\lceil(t+1)/2\rceil$ points of $A$ and $B(\nu(A))\cup \nu(A)=\nu(L\cup R)$.

We are left with the case $B(\nu(A))\subset \nu(L)$, which might happen for any $t\geq 2$. In particular, this gives $B(\nu(A))=\nu(L\setminus A)$.  We remark that this is the only possibility happening in the case $r<t$ or $r=t$ and $t $ even.  
\end{proof}

Now that we characterized emptiness for $\Bb_r(X)$, we want to do the same for $\Ee_r(X)$ and $\tilde{\Ee}_r(X)$. Let us start with some preliminary results.

\begin{lemma}\label{lemma: al piu t pt tutti su retta non danno una tupla in E per SV}
Let $k\ge 2$, $t:= \min\{d_1,\dots,d_k\}$ and let $L\subset Y$ be a curve of multidegree $\varepsilon_i$ with $d_i=t$. Let $A\subset L$ be a set of $r$ points and let $q\in L\setminus A$. 
 If $r\le t$, then $\nu(q)\notin E(\nu(A))$.
\end{lemma}
 \begin{proof}
We remark that since $\Oo_{\PP^{n_i}}(1)$ is very ample, for all $a\in A$ we can find a divisor $H_a\in |\Oo_Y(\varepsilon_i)|$ passing through $a$ and not passing through $q$. If we call  $H:=\cup_{a\in A}H_a$ then we have that $H\in |\Oo_Y(r\cdot \varepsilon_i)|$ contains $A\cap L$ but does not contain $q$. Fix $j\neq i$. By the definition of multidegree of a curve, the set $\pi_j(L)$ is a point and $\pi_j(L)=\pi_j(q)$, so for a $W\in |\Oo_q(\varepsilon_j)|$ we have that $W$ contains $L$. We conclude since $H\cup W$ is smooth at $ q$ and it is singular in $A\cap L$.
\end{proof}

In the next lemma we look at a configuration of $ t+1$ points lying on a curve of multidegree $\varepsilon_i$ for some $i$ such that $d_i=t$. We will show how this configuration of points is an example of points lying in $\tilde{\Ee}_{r+1}(X)$. 
\begin{lemma}\label{lemma: t+1 punti su una retta sono in Ee caso SV}
Let $k\ge 2$, $t:=\min\{d_1,\dots,d_k\}$ and let $L\subset Y$ be a curve of multidegree $\varepsilon_i$ with $d_i=t$. Let $A\subset L$ be a set of $t+1$ points and fix $q\in L\setminus A$. Then $H^0(\Ii_{2A\cup 2q}(d_1,\dots,d_k)) =H^0(\Ii_{2A}(d_1,\dots,d_k))$ and hence $\nu(L\setminus A)\subseteq E(\nu(A))$. 
\end{lemma}
\begin{proof}
Without loss of generality we take $i=1$, so from now $L$ is of multidegree $\varepsilon_1$ where $d_1=t$.

To prove the lemma it is sufficient to prove that $H^0(\Ii_{2A\cup v}(t,d_2,\dots,d_k)) =H^0(\Ii_{2A}(t,d_2,\dots,d_k))$ for every connected degree $2$ zero-dimensional scheme $v$ such that $v_{\red} =q$. 

Since $L$ is of multidegree $\varepsilon_1$, we can write $L $ as $L=\PP^1\times \{ p_2\}\times \cdots \times \{p_k\}$ for some $p_i\in \PP^{n_i}$, and we take $q=(q_1,p_2,\dots,p_k)$ for some $q_1\in \PP^1$. Set $T:= \eta_1^{-1}(\eta_1(q))$ (see \Cref{notation SV}) so that $T =\PP^{n_1}\times \{p_2\}\times \cdots \times \{p_k\}$. We will workout all possibilities for a fixed connected degree $2$ zero-dimensional scheme $v$ such that $v_{\red} =q$ depending on whether it is included in $T$ or not.
\smallskip

\noindent\emph{Case $v\subset T$.} First assume $v\subset L$. Since $L$ is contained in the base locus of $\Ii_{2A}(t,d_2,\dots,d_k)$ the statement is proved in this case. 
Now assume $v\not\subset L$, so in particular we have that $n_1\ge 2$.
Let $Y=\PP^{n_1}\times Y_1$ (see \Cref{notation SV}), we prove the statement by induction on $\dim(Y_1)\ge1$. Assume $\dim(Y_1)=1$, so that $Y=\PP^{n_1}\times \PP^{1}$, $T=\PP^{n_1}\times \{p_2\}$ is a divisor in $|\Ii_{v}(\epsilon_2)|$. Consider the residual exact sequence of $T$ for $2A$:
$$0\to \Ii_{A}(t,d_2-1)\to \Ii_{2A\cup v}(t,d_2)\to \Ii_{(2A\cap T)\cup v, T}(t,d_2)\to0. $$
By \Cref{ex: buono retta doppia con d punti è tangential contact locus} we have
$h^0(\Ii_{(2A\cap T)\cup v, T}(t,d_2))=h^0(\Ii_{(2A\cap T), T}(t,d_2))$.
Since $A\subset L$ and $\deg(A)=t+1$, we have $h^1(\Ii_{A}(t,d_2-1))=h^1(\Ii_{A,L}(t,d_2-1))=0$. Hence we have
$$h^0(\Ii_{2A\cup v}(t,d_2))=h^0(\Ii_{A}(t,d_2-1))+h^0(\Ii_{(2A\cap T)\cup v, T}(t,d_2))=$$
$$=h^0(\Ii_{A}(t,d_2-1))+
h^0(\Ii_{(2A\cap T), T}(t,d_2))\ge h^0(\Ii_{2A}(t,d_2)))$$
and this concludes the proof of the base case of the induction. Now it is easy to complete the inductive step iterating the same argument. 
\smallskip

\noindent\emph{Case $v\nsubseteq T$.} Now assume $v\nsubseteq T$. This is equivalent to assuming the existence of $i\in \{2,\dots,k\}$ such that the scheme-theoretic image $\pi_i(v)$ of $v$ by the submersion $\pi_i$ is not $p_i=\pi_i(q)$ with its reduced structure. In other words, $\pi_i(v)$ is a degree $2$ zero-dimensional scheme supported at $p_i$. Since $\Oo_{\PP^{n_i}}(1)$ is very ample, there is $H\in |\Ii_{p_i}(\epsilon_i)|$ such that $\pi_i(v)\nsubseteq H$. Hence $(2A\cup v)\cap H = (2A\cap H)\cup \{q\}$ and
$\Res_H(2A\cup v) = A\cup \{q\}$. Consider the residual exact sequence of $H$
\begin{equation*}
0\to \Ii_{A\cup \{q\}}(t,d_2,\dots,d_k)(-\epsilon_i)\to \Ii_{2A\cup v}(t,d_2,\dots,d_k)\stackrel{u}{\to} \Ii_{(2A\cup \{q\})\cap H,H}(t,d_2,\dots,d_k)\to 0,
\end{equation*}
and let $\tilde{u}$ be the map induced in the cohomology by $u$. Notice that  $h^1(\Ii_{A}(t,d_2,\dots,d_k))=0$ otherwise we would get a contradiction by \Cref{lemma: famoso per SV} and and clearly $h^1(\Oo_q)=0$. Hence, passing in cohomology, we can construct the following diagram:
\[
  \xymatrix@!R{
                  &  0               & 0                  & 0                      \\
0 \ar[r]  & {H^0(\Ii_{A\cup q}(t,d_2,\dots,d_k)(-\varepsilon_i))} \ar[r]   & {H^0(\Ii_{2A\cup v}(t,d_2,\dots,d_k))} \ar[r]^{\tilde{u}}   & {H^0(\Ii_{(2A\cup q)\cap H,H}(t,d_2,\dots,d_k))}
 \ar@{-}`r[d]`[d]^\delta[d] 
                                                                               & \\
0 \ar[r]  & H^0(\Ii_{A}(t,d_2,\dots,d_k)(-\varepsilon_i))    \ar[r]    & H^0(\Ii_{2A}(t,d_2,\dots,d_k)) \ar[r]      & {H^0(\Ii_{2A\cap H,H}(t,d_2,\dots,d_k))} \ar[r] 
                \ar@{}+<-2.5cm,0cm>="p1"  
                                                                               & 0 \\
0 \ar[r]        & H^0(\Oo_{ q}) \ar[r]
                 \ar@{}[l]+<1cm,0cm>="p2" 
                 \ar`_l[l]+<1.03cm,0cm>`[d]`[d][d]^\delta 
                               & H^0(\Oo_v) \ar[r]   & H^0(\Oo_{q,H}) \ar[r]        & 0 \\
          & {0}  & {0}  & {0}  & 
\ar"1,2";"2,2"   \ar"1,3";"2,3"     \ar"1,4";"2,4"
\ar"2,2";"3,2"   \ar"2,3";"3,3"     \ar"2,4";"3,4"
\ar"3,2";"4,2" \ar"3,3";"4,3" \ar"3,4";"4,4"
\ar"4,2";"5,2"   \ar"4,3";"5,3"     \ar"4,4";"5,4"
\ar@{-}"p1";"p2"|!{"2,3";"3,3"}\hole^\delta
}
\]By the snake lemma we get that $h^1(\Ii_{A\cup q}(t,d_2,\dots,d_k)(-\varepsilon_i))=0$ and so the map $\tilde{u}$ is surjective and this concludes the proof. 
\end{proof}

We are now ready to characterize emptiness of $\Ee_r(X)$ and $\tilde{\Ee}_r(X)$ for any Segre-Veronese variety, including Segre. We point out that, compared to the Veronese case, while the numerical condition characterizing emptiness for the base locus is essentially analogous, the corresponding numerical condition for the strong base locus is shifted by one. 

\begin{theorem}\label{theorem: Ee e tilde Ee per segre-veronesi}
    Fix $k\ge2$ and positive integers $n_i$, $d_i$, for $1\le i\le k$. 
Set $Y:= \PP^{n_1}\times \cdots \times \PP^{n_k}$, $X:= \nu(Y)$ and $t:= \min\{d_1,\dots,d_k\}$. Then 
$$
\Ee_r(X) \text{ and } \tilde{\Ee}_r(X) \text{ are nonempty if and only if } r\geq t+1.
$$
\end{theorem}
\begin{proof}
By \Cref{lemma: t+1 punti su una retta sono in Ee caso SV} and \Cref{remark: una volta in Ee se aggiungo punti resto in Ee} we have that if $r\geq t+1$ then $\tilde{\Ee}_r(X)\neq \emptyset$, so in particular also $\Ee_r(X)\neq \emptyset$.  To conclude the proof of the theorem it is sufficient to prove that $\Ee_r(X)=\emptyset$ for all $r\le t$. 
Take $r\le t$ and assume by contradiction that $\Ee_r(X)\ne \emptyset$. Let $A\in \Ss_r( Y) $ be such that $\nu(A) \in \Ee_r(X)$ and let $q\in Y\setminus A$ with
$\nu(q)\in E(\nu(A))$. Since $\Ee_r(X)\subseteq \Bb_r(X)$, by \Cref{theorem: Bb per SV} we know that $2r\geq t+1  $. Moreover by \Cref{lemma: dato A in Bb esiste retta contenente un po di pt SV} there are $i\in \{1,\dots ,k\}$ with $d_i=t$ and a curve $L$ of multidegree $\epsilon_i$
such that $q\in L$ and $e:= \#(A\cap L)\geq \lceil(t+1)/2\rceil$.
Notice that by \Cref{lemma: al piu t pt tutti su retta non danno una tupla in E per SV} we cannot have that $A\subset L$. Hence we have $A\not\subset L$, so that $e \le r-1$. We prove the result by showing the existence of an element $G$ that is singular in $A$ but not in $q$. Set $B=A\setminus (A\cap L)$ and notice that $\#B=r-e\leq \lfloor(t-1)/2\rfloor$ so $2\#B\leq t-1$. Since $\Oo_{\PP^{n_i}}(\varepsilon_i)$ is very ample, for every $a\in A\cap L$ there is $H_a$ containing $a$ but not containing $q$. Set $H:=\cup_{a\in A}H_a $ and notice that $H\in|\Oo_{Y}(e\cdot \varepsilon_i)| $. Fix now $j\neq i$ and notice that by definition of multidegree of a curve, the set $\pi_j(L)$ is a point with $\pi_j(L)=\pi_j(q)$. The set $B\neq \emptyset$ can be split in the disjoint union $B=B_1\sqcup B_2$, where $B_1=\{ a\in B \, |\, \pi_j(a)=\pi_j(q) \}$ and $B_2=B\setminus B_2$, where eventually one of them can also be empty. For a geneal $W\in |\Ii_q(\varepsilon_j)|$ we have $W\supset L$ and $W\supset B_1$. Hence, so far $H\cup W$ is smooth at $q$, singular at each point of $A\cap L$ and it contains also $B_1$. Now, for each $a\in B_2$ the general $M_a\in |\Ii_a(\varepsilon_j)|$ does not contain $q$. Call $M=\cup_{a\in B_2}2M_a$ and notice that $H\cup W\cup M $ is smooth at $q$, it contains $B_1$ and it is singular at each point of $A\cap L$ and at each point of $B_2$. Lastly, for each $a\in B_1$ take a general $U_a\in |\Oo_Y(1,\dots,1)|$, so in particular $U:=\cup_{a\in B_1}U_a$ does not contain $q$. By construction the divisor $G_1:=H\cup W\cup M\cup U$ is singular in $A$, it is smooth at $q$ and has multidegree $(b_1,\dots,b_k)$ where 
\begin{itemize}
\item $b_i=e+\#B_1\leq e+\#B\leq t$, \item $b_j=1+2\#B_1+\#B_2\leq 1+2\#B\leq t$,
\item $b_k=\#B_1\leq \#B\leq t$, for all $k\notin \{i,j\}$.
\end{itemize}
Since $b_m\leq t_m$ for all $m\in \{ 1,\dots,k\}$, there is $G_2\in |\Oo_Y(d_1-b_1,\dots,d_k-b_k)|$ with $G_2\cap (A\cup \{ q\}) =\emptyset$ and we allow the case $G_2=\emptyset $ if $b_m=d_m$ for all $m$. 
We conclude since we have found a divisor $G:=G_1\cup G_2$ such that $A\subset \sing(G) $ and $q\notin \sing(G)$. 
\end{proof}

Going back again to the content of \Cref{connection terracini}, a comparison between $\Bb_r(X),\Ee_r(X)$ and $\TT_r(X)$ is trickier. In particular, our main \Cref{theorem: Bb per SV} and \Cref{theorem: Ee e tilde Ee per segre-veronesi} work for all Segre-Veronese varieties, including Segre. However, the literature for Terracini loci in the context of Segre varieties is much more limited and, to the best of our knowledge, restricted to \cite{BBS23}. If $X$ is a Segre variety, \Cref{theorem: Bb per SV} and \Cref{theorem: Ee e tilde Ee per segre-veronesi} give $\Ee_2(X)\neq \emptyset$ while $\TT_2(X)=\emptyset$ (cf. \cite[Proposition 2.9]{BBS23}). Considering a Segre-Veronese variety embedded in degree $d_1,\dots,d_k$ where at most one of the $d_i$'s is one, \cite[Theorem 7.10]{GSTT} classifies the first Terracini configuration. This is given by at least $\lceil (t+2)/2\rceil$ points that lie on a curve of multidegree $\varepsilon_i$, for $i$ such that $d_i=t$. Hence, if for instance $r\geq t+1 $ and we consider $r$ points all on a curve of multidegree $\varepsilon_i$ then they provide an example of both $\TT_r(X)$ and $\Ee_r(X)$.

\section{Connection with identifiability}\label{section: identifiability}
Let $X\subset \PP^N$ be an integral and nondegenerate variety. 
Recall that  $\Ss_r(X)$ is the variety parametrizing unordered sets of $r$ smooth points.
In this section we understand how $\Ee_r(X)$ connects with the notion of \emph{identifiability}. We briefly recall the different notions below.
\begin{definition}
A point $q\in \PP^N$ is \emph{identifiable} with respect to $X$ if there is $A\in \Ss_r(X)$ with $q\in \langle A \rangle$ and for any other $B\in \Ss_s(X)$ with $s\leq r$ we have that $q\notin \langle B \rangle $. In other words, $q$ is identifiable if it lies in a unique $r$-secant plane given by points of $X$.

A variety $X\subset \PP^N$ is $r$-\emph{identifiable} if a general element in $\sigma_r(X)$ has a unique expression as a combination of $r$ points of $X$.
\end{definition}
\cite[Proposition 2.4]{COtwd} states that if there exists a set of $r$ particular points $A\in \Ss_r(X)$ such that the span $\langle 2A \rangle$ contains $T_pX$ only if $p\in A$ then $X$ is $r$-identifiable. This result can be rephrased as follows: if  $A\notin \Ee_r(X)$ for some  $A\in \Ss_r(X)$ then  $X$ is $r$-identifiable, or equivalently, if  $\Ee_r(X)\neq \Ss_r(X)$ then $X$ is $r$-identifiable. Hence, to get generic $r$-identifiability for a given $X$ it is enough to find a set of $r$ smooth points not in $\Ee_r(X)$ or even not in $\Bb_r(X)$. 

The following 
statement can also be found
in \cite[Lemma 5.1]{COV17} where it is called \emph{Hessian criterion}, see also \cite{COV14}.  We reformulate it here using the notions of strong base locus and Terracini locus. We recall that given $q\in \PP^N$ the $X$-rank of $q$ is the minimum integer $r$ such that $q\in \langle A\rangle$ for some $A\in \Ss_r(X)$.
\begin{proposition}\label{prop: ovvia}
    Let $X\subset \PP^N$ be integral and nondegenerate. Fix an integer $r\leq \frac{N+1}{\dim X+1}-1$. Let $A\in \Ss_r(X)$. If $A\notin \Ee_r(X)\cup \TT_r(X)$ then any $q\in \langle A \rangle$ with $q\notin \mathrm{Sing}(\sigma_r(X))$  is an identifiable point of $X$-rank $r$.
\end{proposition}

\begin{proof}
First notice that since $r\leq \frac{N+1}{\dim X+1}-1$ then  $h^0(\Ii_{2A}(1))\neq 0$ for any $A\in \Ss_r(X)$. Moreover, since $A\notin \TT_r(X)$ we observe that the $r$-th secant variety of $X$ is not defective.
    Let $A\in \Ss_r(X)$ and take a smooth $q\in \sigma_r(X)$ with $q\in \langle A\rangle $. Since $A\notin \TT_r(X)$ and $h^0(I_{2A}(1))\neq 0$ then $h^1(\Ii_{2A}(1))=0$, which means that $\dim \langle 2A \rangle=r(n+1)-1=\dim T_q\sigma_r(X)$ and in particular, $\langle 2A\rangle=T_q\sigma_r(X)$. 
    Assume by contradiction that there exists $B\in S_r(X)$ with $B\neq A$ such that $q\in \langle B\rangle $. Since $\langle 2B\rangle\subseteq T_q\sigma_r(X)$
for every $x\in B$, we have $T_xX\subset \langle 2A\rangle $ for $x\in B\setminus A$, which is impossible because $A\notin \Ee_r(X)$.  So we have proved that any $q\in A$ has a unique decomposition of length $r$ and in particular this implies that the rank of $q$ is $r$. \end{proof}
Although \Cref{prop: ovvia} is quite simple, making it explicit to determine if a given point (or tensor) $q$ is identifiable is actually challenging. Indeed, given a pair $(q,A)$ with $q\in \langle A\rangle$ for some $A\in \Ss_r(X)$, the conditions $q\notin \TT_r(X)$ and $q\notin \Ee_r(X)$ are easy to verify, especially in the context of tensor related varieties. By contrast, the condition that $q$ is a smooth point of $\sigma_r(X)$ is highly nontrivial to verify, the main obstruction being the fact that the singular locus of a secant variety is not known in general. It is worth mentioning however, that if $X$ is a Veronese variety \cite[Lemma 5.4]{COV17} explicitly gives a range under which it is possible to determine smoothness by relying on the use of local equations (we refer the interested reader also to \cite{landsberg2013equations, BuczynskaBuczynski2014}).    

We conclude with the following consequence of 
\Cref{prop: ovvia} that states that all the possible decompositions of a given $q\in \langle A\rangle$ can be found in the contact locus of $A$.
\begin{corollary}
  Let $X\subset \PP^N$ and fix an integer $r\leq \frac{N+1}{\dim X+1}-1$. Let $A\in \Ss_r(X)$ with $A\notin \TT_r(X)$ and let $q\in \langle A\rangle$ with  $q\notin \mathrm{Sing}(\sigma_r(X))$. Then any other $B\in \Ss_r(X)$ such that $q\in \langle B\rangle$ is contained in $A\cup E(A)$.   
\end{corollary}

\bibliographystyle{alpha}
\bibliography{references.bib}

\Addresses

\end{document}